\documentclass[11pt, oneside]{amsart}   	
\usepackage{geometry}                		
\usepackage{amssymb}
\usepackage{hyperref}
\usepackage{multirow}

\usepackage{todonotes}
\usepackage{amssymb,amsthm,amsmath}
\usepackage{hyperref}
\usepackage{xypic}
\usepackage{tikz}
\usetikzlibrary{calc}
\usepackage{rotating}
\newtheorem{theorem}{Theorem}
\numberwithin{theorem}{section}

\newtheorem{conjecture}[theorem]{Conjecture}
\newtheorem{proposition}[theorem]{Proposition}

\newtheorem{lemma}[theorem]{Lemma}
\newtheorem{corollary}[theorem]{Corollary}

\newtheorem{assumption}[theorem]{Genericity assumption}

\newtheorem*{proposition*}{Proposition}
\newtheorem*{theorem*}{Theorem}
\newtheorem*{lemma*}{Lemma}
\newtheorem*{claim*}{Claim}
\newtheorem*{conjecture*}{Conjecture}

\theoremstyle{definition}
\newtheorem{definition}[theorem]{Definition}

\theoremstyle{remark}
\newtheorem{example}[theorem]{Example}
\newtheorem{remark}[theorem]{Remark}
\newtheorem*{example*}{Example}

\newcommand{\Trop}{\text{Trop}}
\newcommand \RR {\mathbb{R}}  
\newcommand \CC {\mathbb{C}}  
\newcommand{\PSR}{\CC\{\!\{t\}\!\}}
\newcommand \ddiv {\text{div}}

\DeclareMathOperator {\val}{val}
\DeclareMathOperator {\ini}{\mathcal{I}}

\newcommand {\dunion}{\,\mbox {\raisebox{0.25ex}{$\cdot$} \kern-1.83ex $\cup$}
  \,}

\title{Lifting tropical bitangents}
\author{{\larger Y}{\smaller oav} {\larger L}{\smaller en and}  {\larger H}{\smaller annah} {\larger M}{\smaller arkwig}}
\address{Yoav Len, Georgia Institute of Technology, Department of Mathematics}
\email{yoav.len@math.gatech.edu}

\address{Hannah Markwig, Eberhard Karls Universit\"at T\"ubingen, Fachbereich Mathematik}
\email{hannah@math.uni-tuebingen.de}


\begin{document}
\maketitle
\begin{abstract}
We study lifts of tropical bitangents to the tropicalization of a given complex algebraic curve together with their lifting multiplicities.
Using this characterization, we show that generically all the seven bitangents of a smooth tropical plane quartic lift in sets of four to algebraic bitangents. We do this constructively, i.e.\ we give solutions for the initial terms of the coefficients of the bitangent lines. This is a step towards a tropical proof that a general smooth quartic admits $28$ bitangent lines. The methods are also appropriate to count real bitangents; however the conditions to determine whether a tropical bitangent has real lifts are not purely combinatorial.
\end{abstract}

\section{Introduction}
%
%

This paper is concerned with combinatorial aspects of bitangents to plane curves. Our main result is a classification of algebraic bitangent lines tropicalizing to a given tropical line. Such lines are called \emph{lifts}.  
The solutions are concrete in the sense that we find the initial terms of the bitangents explicitly. 
As a result, our methods are easily adaptable to real Puiseux series; however, as far as we can tell the problem is not purely combinatorial in the real case. 

Our motivation stems mostly from the recent study of smooth tropical plane quartics and their relation with complex algebraic quartics. 
In 1834, Pl\"ucker showed that every smooth complex plane quartic admits $28$ bitangent lines \cite{Plucker}. Tropical quartics exhibit a similar behavior to their complex algebraic counterparts, but  admit 7 rather than 28 tropical bitangent lines \cite{BLMPR}. More precisely, tropical lines intersecting a tropical quartic twice with multiplicity two can be grouped into natural equivalence classes of which there are exactly 7.

The existence of 7 tropical bitangent classes raises the question whether each class has representatives which are the tropicalization of a complex algebraic bitangent, and if so, how many, and how many lifts each such representative has. It is natural to conjecture that every class of tropical bitangents provides exactly 4 lifts. Chan and Jiradilok  \cite{Chan2} studied $K4$ tropical quartics (i.e.\ tropical quartics whose underlying graph is a complete graph on 4 vertices after contracting ends and leaves), and provided a positive answer in that case. Our main result provides a positive answer for any generic tropically smooth quartic.

\begin{theorem*}[See Theorem \ref{thm-quartics}]
\label{thm:main}
Let $C$ be a smooth plane algebraic quartic defined over the complex Puiseux series, whose tropicalization is a generic smooth tropical quartic $\Gamma$. Then every equivalence class of tropical bitangents to $\Gamma$ lifts to $4$ different bitangents of $C$.
\end{theorem*}
Interestingly, there are equivalence classes such that one representative lifts four times, or such that two representatives each lift twice, or such that one lifts twice and two once, or such that four representatives each lift once, but there is no class with one representative that lifts three times and one that lifts once.

Our methods of proof include combinatorial classifications, elementary solving of systems of equations for initials, tropical modifications and re-embeddings (also referred to as tropical refinement) and a variant of the Patchworking method. While no individual ingredient in this list of methods is new, some tools had to be sharpened for our purpose, and different tools had to be coordinated and matched to allow our intended combination.
Parts of the proof are technical and involve a case-by-case analysis. The fact that the methods nevertheless suffice to compute lifts for all tropical bitangents can be viewed as a proof-of-concept in tropical computation. Computing lifts of tropical objects in the field of Puiseux series is, as tropical computation in general, symbolic computation. As many of the tropical papers in the Journal of Symbolic Computations do, we provide an algorithm in tropical computation, it outputs lifts of tropical bitangents.

As a consequence of Theorem  \ref{thm-quartics}, we partially recover Pl\"ucker's result.
\begin{corollary}
A generic tropically smooth plane algebraic quartic has 28 bitangent lines.
\end{corollary}
\noindent We believe that the techniques of this paper can be generalized to tropicalizations of quartics that do not have a plane tropically smooth model. This would yield  a tropical proof of Pl\"ucker's theorem.

\vspace{0.2cm}

Our methods are not restricted to the case of plane quartics, and extend to curves of degree $d\geq 4$.
The formulas describing the lifting multiplicities of bitangents to arbitrary plane curves can be found in Theorem \ref{Thm:Lifting}, Proposition \ref{prop-starshaped} and Lemma \ref{lem-noliftforpointsinsegment}.

\begin{theorem*}[See Theorem \ref{Thm:Lifting}]
Let $C$ be a generic smooth algebraic plane curve defined over the complex Puiseux series tropicalizing to $\Gamma$, and let $\Lambda$ be a tropical line  that is bitangent to $\Gamma$ at two points. Then there is an explicit formula, in terms of the local combinatorics at the tangency points, for the number of lifts of $\Lambda$ to bitangent lines to $C$.
\end{theorem*}

\vspace{0.2cm}

Modern proofs for the existence of the 28 bitangents  use  the correspondence between bitangents and odd theta characteristics. A result due to Zharkov \cite{Zh} shows that an abstract tropical curve of genus $g$ admits $2^{g}-1$ effective theta characteristics. 
When the tropical curve arises as the skeleton of an algebraic curve, each of these can be lifted to $2^{g-1}$ odd theta characteristics \cite{JL} (resp.\ \cite{Panizzut} for hyperelliptic curves). Since equivalence classes of bitangents to tropical quartics are in bijection with the effective theta characteristics,  it follows that each of them lifts to four algebraic bitangents. 
However, the study of theta characteristics leaves to discover which representatives within the equivalence class lift how many times, and what are concrete solutions for the initials of the coefficients of the defining equations. Moreover, it does not provide insight about tangent lines to curves of higher degree. Our paper provides answers concerning these problems.

Our techniques can be straight-forwardly adapted to real Puiseux bitangents to generic smooth algebraic plane curves defined over the real Puiseux series. In the real case, the number of bitangents depends on the topology of the curve. There can be $4$, $8$, $16$, or $28$ bitangents. The conditions governing the number of real lifts are not purely combinatorial however: they involve signs of coefficients of the defining equation after coordinate changes. 
While our methods lay the ground for a tropical approach to count real bitangents, we believe that a major combinatorial classification task (in which one also needs to control the coefficients after suitable coordinate changes) has to be completed before tropical geometry becomes an efficient tool to count real bitangents. We leave this combinatorial classification task for further research.

The paper is organized as follows. In Section \ref{sec-prelim}, we introduce all necessary preliminaries, including tropicalizations of curves and their intersections, Wronskians and Patchworking and tropical modification and re-embeddings. In Section \ref{sec-lifting}, we study lifts of tropical bitangents to curves of any degree. In Section \ref{sec-quartics}, we apply these methods to provide a full classification of lifts of tropical bitangents for generic tropicalizations of quartics. In Section \ref{sec-realbitangents}, we compute all real lifts of a certain tropicalized quartic defined over the reals. With that comprehensive example, we illustrate the applicability of our lifting methods to the real case. 

\subsection{Acknowledgements}
The authors gratefully acknowledge partial support by DFG-grant MA 4797/6-1.
We acknowledge the hospitality of the Fields Institute, where much of this work has been carried out. 
We would like to thank Nathan Ilten, Thomas Markwig, Eugenii Shustin and Ilya Tyomkin for valuable discussions and suggestions. In addition, we thank Heejong Lee, Run Tan,  Akshay  Tiwary and Wanlong Zheng for helpful discussion.
We are thankful to the anonymous referees whose comments  helped to improve  previous versions of this paper.
Computations have been made using \textsc{Singular} \cite{DGPS}, mainly the \textsc{Singular}-library tropical.lib \cite{JMM07a} and the Sage-package Viro.sage  \cite{Viro:Sage} resp.\ the online tool \begin{displaymath}\mbox{https://www.math.uni-tuebingen.de/user/jora/patchworking/patchworking.html}.\end{displaymath}  for combinatorial patchworking developped by Boulos El-Hilany, Johannes Rau and Arthur Renaudineau.

\section{Preliminaries}\label{sec-prelim}
We assume that the reader is familiar with tropical plane curves (i.e.\ balanced weighted rational planar graphs) arising as bend loci of tropical bivariate polynomials and their dual Newton subdivisions. For an introduction to this topic, see e.g.\ \cite{FirstSteps, Gathmann}. 
Figure \ref{fig-bigquartic} shows an example of a tropical plane curve defined by a quartic tropical polynomial, and its dual Newton subdivision.

%

\subsection{Intersection multiplicity} Suppose that two tropical plane curves $\Gamma_1,\Gamma_2$ intersect properly (i.e.\ in finitely many points which are not vertices), and $P\in \Gamma_1\cap \Gamma_2$. Choose weighted direction vectors $u_1,u_2$ for edges of $\Gamma_1,\Gamma_2$ emanating from $P$. Then the \emph{intersection multiplicity} of the two curves at $P$ is $\Gamma_1\cdot\Gamma_2\mid_P = |\det(u_1,u_2)|$, and $\Gamma_1\cdot \Gamma_2=\sum_{P\in \Gamma_1\cap \Gamma_2}\Gamma_1\cdot\Gamma_2\mid_P$. The balancing condition ensures that this definition does not depend on the choice of direction vectors. If the curves do not intersect properly, we  use the \emph{stable} intersection: choose a direction $v$ so that $\Gamma_1$ and $\Gamma_2 + \epsilon\cdot v$ intersect properly whenever $\epsilon$ is small enough. Then the stable intersection is
$$
\lim_{\epsilon\to 0} \Gamma_1\cdot(\Gamma_2+\epsilon\cdot v).
$$

A plane tropical curve is \emph{smooth} of degree $d$ if it is dual to a unimodular triangulation of the triangle with vertices $(0,0), (0,d), (d,0)$. Here, unimodular means that it is subdivided to $d^2$ triangles each having area $\frac{1}{2}$. 
%

We assume that the reader is familiar with the theory of tropical divisors on curves. For a lucid introduction to the theory of tropical divisors see \cite{BJ}.

\begin{definition}
Let $D$ be the stable intersection of a tropical line $\Lambda$ and a tropical plane curve $\Gamma$. We say that $\Lambda$ is \emph{bitangent} to $\Gamma$ at $P$ and $Q$ (or that $(\Lambda,P,Q)$ is a bitangent) if there exists a piecewise linear function $f$ that is constant outside of the intersection, and such that $D+\ddiv(f)$ is effective and contains $2P+2Q$. The points $P$ and $Q$ are called tangency points.
\end{definition}
\noindent For instance, when a component of $\Lambda\cap\Gamma$ is an edge, there is a tangency at the midpoint of the edge. By abuse of notation, when the points P and Q are known from context, we may omit them, and refer to $\Lambda$ as a bitangent.

It will often be the case that a tropical curve admits infinitely many bitangents. We therefore declare that two bitangents are equivalent if we can continuously move one line to the other while maintaining bitangency. For quartics, this is equivalent to saying that the bitangents correspond to the same theta characteristic in the tropical Jacobian \cite[Definition 3.8]{BLMPR}. For example, in Figure \ref{fig-bigquartic}, if we draw a tropical line with a vertex at the point $2a$ and then move the vertex continuously until it reaches the point $2b$, we obtain a family of equivalent tropical bitangents. We will see later that the liftable members of this family are precisely the two lines with vertices $2a$ and $2b$, and both have two distinct lifts. 

From \cite{Mor15}, the tropicalization of every algebraic bitangent line is a tropical bitangent line. However, as we will see, there are many tropical bitangents that are not liftable to algebraic bitangents (this should not come as a surprise, considering that there may be infinitely many tropical bitangents!). However, every equivalence class of bitangents has at least one liftable representative.

\subsection{Tropicalization of plane curves and their intersections}\label{subsec-intersect}

By the famous theorem of Kapranov, tropical plane curves arise as tropicalizations of curves defined over a field with a non-Archimedean valuation \cite{EKL06}. Due to our computational approach and for the sake of explicitness, we let $K=\CC\{\{t\}\}$ be the field of Puiseux series, with 
 valuation $\val$ sending a Puiseux series to its least exponent, and $0$ to $\infty$. For a Puiseux series $a$ of valuation $0$, we denote by $\overline{a}$ its residue in $\CC$. Explicitly, if 
\[
a = a_0 + a_1^{\frac{1}{n}} + a_2^{\frac{2}{n}} + \ldots,
\]
then $\overline{a} = a_0$.

For a point $(x,y)$ in the two-dimensional torus $(K^\ast)^2$, we call $(-\val(x),-\val(y))\in \mathbb{R}^2$ its \emph{tropicalization}.
Vice versa for a point $(X,Y)\in \mathbb{R}^2$, we call any point $(x,y)\in~(K^\ast)^2$ satisfying $-\val(x)=X$, $-\val(y)=Y$ a \emph{lift} of $(X,Y)$.
Let $C$ be a plane curve defined by the polynomial $q$ over $K$. The \emph{tropicalization} of $C$, denoted $\Trop(C)$, is the closure with respect to the Euclidean topology of the set of tropicalizations of the torus points of $C$. Equivalently, it is the set of points $(X,Y)\in\mathbb{R}^2$ in which the maximum value of $-\val(c) + aX + bY$ is obtained  at least twice for monomials $cx^ay^b$ of $q$. From the Newton polygon and the regular subdivision induced by the negative valuations of the coefficients of $q$, the tropicalization obtains the structure of a plane tropical curve, with weights given by the lattice lengths of the dual edges in the Newton subdivision. Any curve $C$ tropicalizing to a given plane tropical curve $\Gamma$ is called a \emph{lift} of $\Gamma$.

For example, the tropical plane quartic depicted in Figure \ref{fig-bigquartic} has a lift defined by the polynomial $q\in \CC\{\{t\}\}[x,y]$:
\begin{align}\begin{split}q=&t^{11}\cdot x^4+3t^4\cdot x^3y+t^3\cdot x^2y^2+t^5\cdot xy^3+t^{15}\cdot y^4+t^8\cdot x^3+x^2y+xy^2+t\cdot y^3\\&+t^6\cdot x^2+xy+t^3\cdot y^2+t^5\cdot x+t^7\cdot y+t^{13}.\end{split}\label{eq-exq}\end{align}

Analogously, we can define tropicalizations $\Trop(V(I))$ of subvarieties $V(I)$ of the $n$-dimensional torus $(K^\ast)^n$ defined by an ideal $I\subset K[x_1,\ldots,x_n]$ by applying the negative valuation coordinatewise.

Let $C_1$ and $C_2$ be two plane curves over $K$ that do not share a component.
The tropicalization of the finite set $C_1\cap C_2$ is contained, but is not necessarily equal to the (set-theoretic) intersection of $\Trop(C_1)$ and $\Trop(C_2)$. In particular, $\Trop(C_1\cap C_2)\subsetneq\Trop(C_1)\cap \Trop(C_2)$ if the latter intersection is not proper. It follows from \cite{Mor15} that for a given non proper intersection of two tropical plane curves $\Gamma_1$ and $\Gamma_2$, we can find lifts $C_1$ and $C_2$ such that $\Trop(C_1\cap C_2)$ equals a given divisor $D$ of $\Gamma_1$ and $\Gamma_2$ only if $D$ is linearly equivalent to the stable intersection divisor, via a rational function that is constant outside the region of intersection.
From this, we can deduce where a point of tangency can be in a non-proper intersection of a tropical line $\Lambda$ with a tropical curve $\Gamma$ (see Figure \ref{fig-3cases}):
\begin{enumerate}
\item If $\Lambda$ meets $\Gamma$ in a line segment with stable intersection of multiplicity $2$, the tropicalization of a tangency point has to be the midpoint of the line segment.
\item If a connected component of $\Lambda\cap\Gamma$ consists of a vertex and $3$ edges in the standard directions (so it locally looks like a tropical line), the total intersection multiplicity equals $4$ and we can have two points of tangency. If the minimal lengths of the three adjacent edges is $l_1$, the two points of tangency appear on the other two edges, at distance $(l-l_1)/2$ from the vertex is the length is $l$ (see also Figure \ref{fig-starshaped}).  We refer to such an intersection as ``star-shaped''.
\item If $\Lambda$ meets $\Gamma$ in a line segment with stable intersection of multiplicity at least $k\geq 4$, the line segment must contain the vertex $P$ of $\Lambda$, and the intersection multiplicity at $P$ is $k-1$ (see also Figure \ref{fig-noliftforpointsinsegment}). 
\end{enumerate}

\begin{figure}
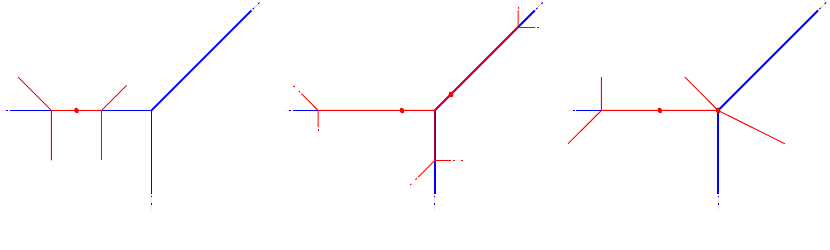
\caption{Positions for tangency points in a non-proper intersection of a tropical line $\Lambda$ with a tropical curve $\Gamma$.}\label{fig-3cases}
\end{figure}

\subsection{How to lift bitangents lines?}\label{sec-wronskian}
To determine a bitangent line $V(\ell)$ lifting $\Lambda$, we first set up a system of equations in which the indeterminates are the coordinates of the point of tangency and the coefficients of the line. We solve for the initial terms of the Puiseux series coefficients of $\ell$. For some cases, there will be multiple solutions for the initial terms, which correspond to a higher multiplicity of a lift of the bitangent. 
We then complete the initial term uniquely to a Puiseux series solving our equations using the following theorem, commonly known at the non-Archimedean implicit function theorem or the multivariate Hensel's lemma \cite[Exercises 7.25, 7.26]{Eisenbud}. It can also be viewed as part of the more general Patchworking procedure as described in \cite[Chapter 2]{IMS09}.

\begin{theorem} \label{Thm:IFT}
Let $f_1,\ldots,f_n$ be elements in $K[x_1,\ldots,x_n]$, and let $f=(f_1, f_2\ldots,f_n)$. Suppose that the valuations of the coefficients of all $f_i$ are nonnegative, and that for each $i$ there is a coefficient of $f_i$ which has valuation zero.
Denote by $J_f$  the Jacobian matrix of $f$. Suppose that $\overline{f(a)} =0$ for some $a=(a_1,\ldots, a_n)$ in $\CC^n$, and that $\overline{\det(J_f(a))}$ is invertible in $\CC$.  Then there exists a unique $b=(b_1,\ldots,b_n)$ in $K^n$ satisfying
\[
f(b)=0,\;\val(b_i)=0\mbox{ and }\overline{b_i}= a_i.
\]
\end{theorem}

Let $C$ be defined by a polynomial $q\in K[x,y]$.
We denote by $(\tilde x,\tilde y)$ a point of tangency, and by $\ell=y+m+nx$ the equation of a line with parameters $m,n\in K$.
The point $(\tilde x,\tilde y)$ has to be a common point of $C=V(q)$ and $V(\ell)$, so it satisfies 
\begin{equation}\label{eq-q}
q(\tilde x,\tilde y)=0
\end{equation}
and \begin{equation}\label{eq-l}
\ell(\tilde x,\tilde y)=\tilde y+m+n\tilde x=0.
\end{equation}
Furthermore, the line defined by $\ell$ should be tangent at $(\tilde x,\tilde y)$, so the tangent vector $(\partial q/\partial x, \partial q/ \partial y)$ at $(\tilde x,\tilde y)$ should be parallel to $(n,1)$. 
This can be rephrased by saying that  the \emph{Wronskian} 
\begin{equation}\label{eq-W}
W(\tilde x,\tilde y,m):=\det\begin{pmatrix}
\partial q/\partial x(\tilde x,\tilde y) & n\\
\partial q/ \partial y(\tilde x,\tilde y) & 1
\end{pmatrix}
\end{equation}
vanishes.

Applying the same argument for another tangency point yields $3$ more equations, giving in total  $6$ equations in $6$ variables  (the two coefficients of the line, and the four coordinates of the two tangency points). We solve these equations for initial terms, then use Theorem \ref{Thm:IFT} to lift the solution to Puiseux series by checking that the determinant of the Jacobian is invertible. The calculation is made simple due to two key observations. First, for each tangency point, the reductions of the equations to initial terms only depend on the terms corresponding to the vertices of the polygon dual to the tropical tangency point. Second, when the two tropical tangency points are in different ends of the line, then (possibly after a change of variables) the set of equations of each point only depend on the coordinates of the point, and either $m$, $n$ or $\frac{m}{n}$. So our system of equations breaks into two systems of $3$ equations in $3$ variables. These observations imply that our computation can be done entirely locally.


\subsection{Tropical modifications and re-embeddings of plane curves}\label{subsec-modification}
Tropical modifications (sometimes also referred to as tropical refinements) appeared in \cite{Mi06} 
and have since been instrumental in various applications, see e.g.\ \cite{AR07, ABBR152, ABBR15, BLdM11, CM14, IMS09, MMS09, Sha10}.  
Here, we consider modifications of the plane
$\RR^2$ along tropical lines (the bitangent lines of the tropicalizations of plane curves).

Modification is a geometric technique for dealing with  the vexing fact that tropicalization depends on the embedding. Some geometric features can be seen in the tropicalization with respect to certain embeddings but not others. 
The quintessential example in our context is when two curves meet at a point, but their tropicalizations meet along a segment (resp.\ tropical segment). Using modifications we may re-embed the curves so that their tropicalizations meet at the expected number of points. 

Modifications and re-embeddings fit more broadly within the context of Berkovich theory and its relation to tropical geometry. 
Unlike tropicalization, the Berkovich analytification, which is the inverse limit of all tropicalizations \cite{Pay09}, does not depend on the chosen embedding. Philosophically this means that for a given ``bad'' tropicalization, we can add functions, embed our curve into bigger and bigger spaces, eventually coming ``close enough'' to the analytification to observe the features we care for. One of the main tools for getting closer to the analytification is via modification. 
The relation with Berkovich theory is not a necessary tool in our paper --- and we therefore refrain from referring to it in the text to make it accessible to readers without background in Berkovich theory.

\begin{remark}\label{rem-coordinatechange}
We will see below that we can describe the tropicalization of a re-embedded curve inside a modified tropical plane by means of coordinate projections. Passing from one coordinate chart to another amounts to applying a linear change of coordinates to the original plane curve. If we can solve a lifting problem for a tropical bitangent after a linear coordinate change, we can of course equivalently solve the original problem. In this way, we will make use of modifications in our problem of lifting tropical bitangent lines. 
\end{remark}

Let $L=\max\{M,N+ X, Y\}$ be a linear tropical polynomial with $M$ and $N$ in $\mathbb{T}_\mathbb{R}:=\mathbb{R}\cup\{-\infty\}$. The graph of $L$
considered as a function on $\RR^2$ consists of at most three linear
pieces. At each break line, we attach two-dimensional cells spanned in
addition by the vector $(0,0,-1)$ (see e.g.\ \cite[Construction 3.3]{AR07}).
 We assign multiplicity $1$ to each cell and obtain a
balanced fan in $\RR^3$. It is called the \emph{modification of
  $\RR^2$ along $L$}. If $\Gamma\subset \mathbb{R}^2$ is a plane tropical curve, we bend it analogously and attach downward ends to get the modification of $\Gamma$ along $L$, which now is a tropical curve in the modification of $\RR^2$ along $L$.

Let $\ell=m+nx+y\in \PSR[x,y]$ be a lift of $L$, i.e.\ $-\val(m)=M$,
and $-\val(n)=N$.  We fix an irreducible polynomial $q\in
\PSR[x,y]$ defining a curve in the torus $(\PSR^*)^2$. The
tropicalization of the variety defined by $I_{q,\ell}=\langle q, z-\ell\rangle\subset \PSR[x,y,z]$
is a tropical curve in the modification of $\RR^2$ along $L$. We call
it the \emph{linear re-embedding} of the tropical curve $\Trop(V(q))$
\emph{with respect to $\ell$}.

For almost all lifts $\ell$, the linear re-embedding equals the
modification of $\Trop(V(q))$ along $L$, i.e.\ we only bend $\Trop(V(q))$ so
that it fits on the graph of $L$ and attach downward ends. 
However, for some choices of lifts $\ell$, the part of $\Trop(V(I_{q,\ell}))$
in the cells of the modification attached to the graph of $L$ contains
more attractive features.  We are most interested in these special
linear re-embeddings. This phenomenon is illustrated in Examples \ref{ex-modi1} and \ref{ex-modi2}.
We distinguish two main cases:
\begin{itemize}
\item[(I)] The tropical line at which we modify the plane is degenerate in the sense that it does not have a vertex. Without restriction we assume it is the horizontal line $\{Y=0\}$.
\item[(II)] The tropical line at which we modify has a vertex. Without restriction we assume that the vertex is at the point $(0,0)$.
\end{itemize}

 In case (I), the modification
  is induced by the tropical polynomial $L=\max\{Y,0\}$.  Its lifts are of the form $\ell=y+m $
  where $m\in \PSR$ has valuation 0. The modified plane contains three
  maximal cells:
\[
\sigma_1=\{Y\leq 0, Z=0\}, \quad \sigma_2=\{Y\geq 0, Z=Y\}, \quad \mbox{ and } \quad
\sigma_3=\{Y=0, Z\leq 0\}. 
\]
By construction, $\sigma_3$ is the unique cell of the modification of $\RR^2$
attached to the graph of $L$. 

We describe $\Trop(V(I_{q,\ell}))$ by means of two projections:
\begin{enumerate}
\item the projection $\pi_{XY}$ to the coordinates $(X,Y)$ produces the
  original curve $\Trop(V(q))$.
\item the projection $\pi_{XZ}$ gives a new tropical plane curve
  $\Trop(V(\tilde{q}))$ inside the projections of the cells $\sigma_2$ and $\sigma_3$, where
  $\tilde{q}=q(x,z-m)$. The polynomial $\tilde{q}$ generates the
elimination ideal $I_{q,\ell}\cap \PSR[x,z]$.
\end{enumerate}

\begin{example}\label{ex-modi1}
Let $q=t+(1+t^2)\cdot x+xy+tx^2\in K[x,y]$. Let $\ell=y+1$. Figure \ref{fig-modi1} shows $\Trop(V(I_{q,\ell}))$ and the two projections $\Trop(V(q))$ and $\Trop(V(\tilde{q}))$, where $\tilde{q}=q(x,z-1)=x+t^2x+xz+tx^2$. The Newton subdivisions of $q$ and $\tilde{q}$ are depicted in Figure \ref{fig-modiNewt1}. 
The choice of lift $\ell$ was made in order to produce cancellation in the projection $\tilde{q}$. For generic lifts, i.e.\ for $l'=y+m$ with $m\in K$ of valuation $0$ but with residue not equal to $1$, we would obtain the modification of the plane curve, whose projection to the $X$ and $Z$-coordinates is depicted in the grey shaded area.

\end{example}

\begin{figure}
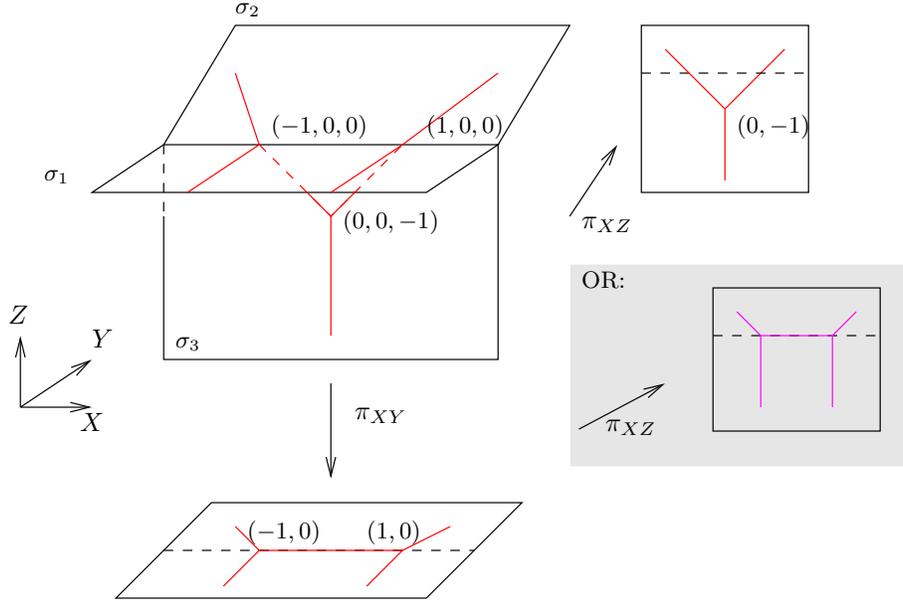
\caption{A tropical curve on a modified plane as in case (I) together with two projections. In the grey shaded area, the projection of the modification of the curve, this is what we get for generic lifts of the tropical function we modify at. Newton subdivisions are depicted in Figure \ref{fig-modiNewt1}.}\label{fig-modi1}
\end{figure}

 In case (II), the modification
  is induced by the tropical polynomial $L=\max\{X,Y,0\}$.  Its liftings are of the form $\ell=y+nx+m $
  where $n,m\in \PSR$ have valuation 0. The modified plane is a standard tropical plane in $\RR^3$. It contains six maximal cells:
\begin{align*}
& \sigma_1=\{X\leq 0,Y\leq 0, Z=0\}, \quad \sigma_2=\{X\geq 0,X\geq Y, Z=X\}\quad\sigma_3=\{X\leq Y,Y\geq 0, Z=Y\}\\ & \sigma_4=\{X\leq 0,Y=0, Z\leq 0\}, \quad \sigma_5=\{Y\leq 0,X=0, Z\leq 0\} \mbox{ and } 
\sigma_6=\{X=Y\geq 0, Z\leq X\}. 
\end{align*}
Now $\sigma_4$, $\sigma_5$ and $\sigma_6$ are the cells of the modification of $\RR^2$ attached to the graph of $L$. 

We describe $\Trop(V(I_{q,\ell}))$ by means of three projections:
\begin{enumerate}
\item the projection $\pi_{XY}$ to the coordinates $(X,Y)$ produces the
  original curve $\Trop(V(q))$.
\item the projection $\pi_{ZY}$ gives a new tropical plane curve
  $\Trop(V(q_{yz}))$ inside the projections of the cells $\sigma_3$, $\sigma_4$ and $\sigma_6$, where
  $q_{yz}=q(\frac{1}{n} z-\frac{1}{n} y+\frac{m}{n},y)$. The polynomial $q_{yz}$ generates the
elimination ideal $I_{q,\ell}\cap \PSR[y,z]$.
\item the projection $\pi_{ZX}$ gives a new tropical plane curve
  $\Trop(V(q_{xz}))$ inside the projections of the cells $\sigma_2$, $\sigma_5$ and $\sigma_6$, where
  $q_{xz}=q(x,z-nx-m)$. The polynomial $q_{xz}$ generates the
elimination ideal $I_{q,\ell}\cap \PSR[x,z]$.
\end{enumerate}

\begin{example}\label{ex-modi2}
Let $q=t^3\cdot x^2y^2+x^2y+xy^2+t\cdot x^2+xy+t^3\cdot y^2\in K[x,y]$. Let $\ell=x+y+1$. Figure \ref{fig-modi2} shows $\Trop(V(I_{q,\ell}))$ and the three projections $\Trop(V(q))$, $\Trop(V(q_{xz}))$ and $\Trop(V(q_{yz}))$, where \begin{align*}q_{xz}=&q(x,z-x-1)=t^3\cdot x^4+(-2t^3)\cdot x^3z+t^3\cdot x^2z^2+(2t^3)\cdot x^3+(-2t^3-1)\cdot x^2z\\&+xz^2+(2t^3+t)\cdot x^2+(-2t^3-1)\cdot xz+t^3\cdot z^2+2t^3\cdot x+(-2t^3)\cdot z+t^3\end{align*} and \begin{align*}q_{yz}=q(z-y-1,y)=&t^3\cdot y^4+(-2t^3)\cdot y^3z+t^3\cdot y^2z^2+(2t^3)\cdot y^3+(-2t^3-1)\cdot y^2z+yz^2\\&+(2t^3+t)\cdot y^2+(-2t-1)\cdot yz+t\cdot z^2+2t\cdot y+(-2t)\cdot z+t .\end{align*} The Newton subdivisions of $q$, $q_{xz}$ and $q_{yz}$ are depicted in Figure \ref{fig-modiNewt2}.
\end{example}

\begin{figure}
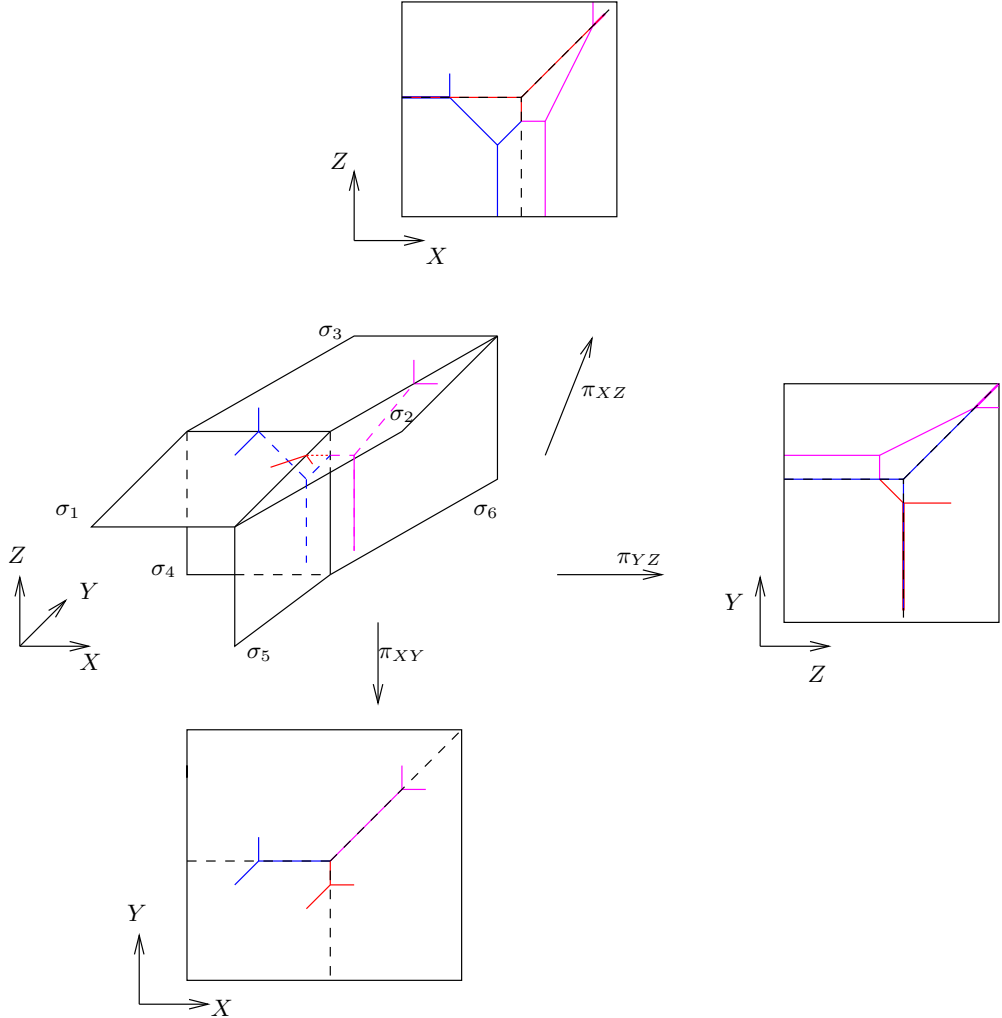
\caption{A tropical curve on a modified plane as in case (II) together with the three projections. Newton subdivisions are depicted in Figure \ref{fig-modiNewt2}.}\label{fig-modi2}
\end{figure}

By Lemma 2.2 of \cite{CM14} (and a slight generalization to cover case (II)), the projections above define $\Trop(V(I_{q,\ell}))$. The content of the lemma is to recover the parts of $\Trop(V(I_{q,\ell}))$ which are not contained in the interior of cells $\sigma_i$ --- for the images of these lower codimension cells the preimage under the projection is not unique.
The projections are given by linear coordinate changes of the original curve. We can study the Newton subdivision of the projected curve in terms of these coordinate changes.

In case (I), a term $a\cdot x^i y^j$  of $q$ is replaced by $x^i(z-m)^j$, and so it contributes to all terms of the form $x^i z^k$ for $0\leq k\leq j$. This is called the ``feeding process'' and is visualized in Figure \ref{fig-modiNewt1}.
From the feeding, we can deduce expected valuations of the coefficients. The subdivision corresponding to the expected valuations is dual to the projection of the modified curve. We care for cases in which there is cancellation and the expected valuation is not taken. These are the special re-embeddings that will make hidden geometric properties of $\Trop(V(q))$ visible.

\begin{figure}
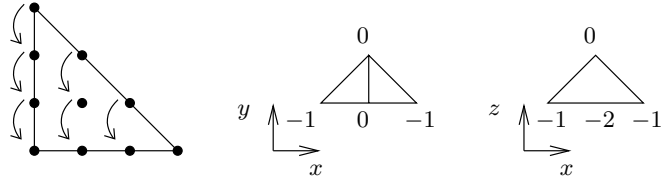
\caption{On the left, a visualization of the ``feeding process'' in case (I) producing $\tilde{q}$. On the right, the Newton subdivisions of the curves $q$ and $\tilde{q}$ from Example \ref{ex-modi1}, the number next to each lattice point denotes the negative valuation of the coefficient. The initial of the coefficient of $x$ was cancelled via feeding from $xy$.}\label{fig-modiNewt1}
\end{figure}

In case (II), a term $a\cdot x^i y^j$ is replaced by $a\cdot (\frac{1}{n} z-\frac{1}{n} y+\frac{m}{n})^i y^j$ in the projection onto the $YZ$ plane, and by $a\cdot x^i (z-nx-m)^j$ in the projection onto the $XZ$ plane. We call this ``cone feeding'', the process is visualized in Figure \ref{fig-modiNewt2}.
Again, we are interested in cases where the valuations of the coefficients after the coordinate change are not as expected.

\begin{figure}
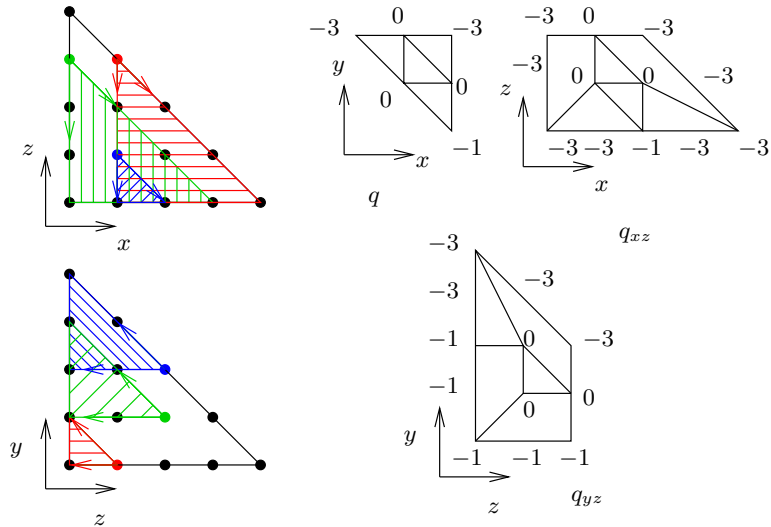
\caption{On the left, a visualization of the ``cone feeding'' in case (II) producing $q_{xz}$, resp.\ $q_{yz}$. In each picture, the cones of three terms are exemplarily shown, corresponding to the red, green and blue lattice point. On the right, the Newton subdivisions of the curves $q$, $q_{xz}$, and $q_{yz}$ from Example \ref{ex-modi2}, the integer next to each lattice point denotes the negative valuation of the coefficient. The initial of the coefficient of $x$, $x^2$ and $x^3$ resp.\ $y$, $y^2$ and $y^3$ were cancelled via feeding from $xy$, $x^2y$ and $xy^2$.}\label{fig-modiNewt2}
\end{figure}

\begin{remark}
To keep the tropicalization of a point of tangency finite (i.e.\ in the tropicalization of the torus), we only modify at lifts of tropical lines that do not contain the tangency point. In order to pick a lift that contains the tangency point, one should extend tropicalization to the boundary (see \cite{Kajiwara,Pay09,Sha10}).
\end{remark}

\section{Lifting bitangent lines}\label{sec-lifting}
In this section, we find the lifting multiplicity of tropical bitangents, according to the combinatorics of the tangency points. We treat the cases where the tangency points are on one or two connected components of the intersection separately. Let $(\Lambda,P,Q)$ be a tropical tangent line. We assume throughout this section that $P$ and $Q$ are different points.

\subsection{Lifting tropical lines tangent at two points of different connected components} 

Let $C$ be a smooth algebraic plane curve tropicalizing to $\Gamma$. We require that $\Gamma$ is smooth, and in addition that $C$ is generic in a sense specified below. Moreover,  the vertices and intersection points of $\Lambda$ and $\Gamma$ are assumed to be integral points (we can always reach that by replacing $t^{\frac{1}{N}}$ by $t$ for some $N$). 

\begin{figure} [h]
\centering
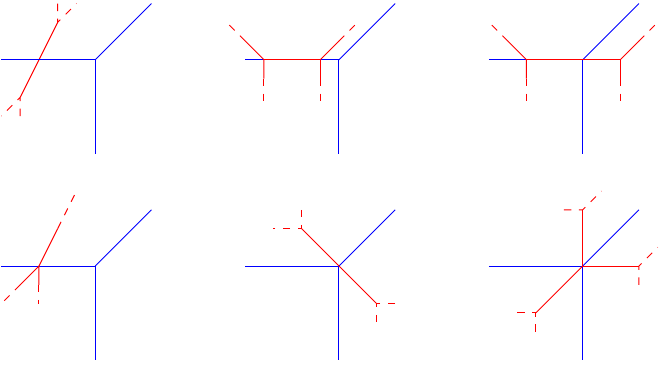
\caption{The possible intersection types --- cases (1)-(5)}
\label{fig-alltypes}
\end{figure}

The possible combinatorial types of intersections of $\Lambda$ and $\Gamma$ (see Figure \ref{fig-alltypes}) of multiplicity greater or equal to two are:
\begin{enumerate}
\item An edge of $\Gamma$ intersecting an edge of $\Lambda$ transversally in a point $P$. 
\item A vertex of $\Gamma$ intersecting the interior of an edge of $\Lambda$ in a point $P$. 
\item An edge of $\Gamma$ intersecting $\Lambda$ non-transversally in a line segment.
\item The interior of an edge $e$ of $\Gamma$ intersecting the vertex $P$ of $\Lambda$. 
\item A vertex of $\Gamma$ intersecting the vertex of $\Lambda$.
\end{enumerate}

In the following theorem, we treat each tangency point separately, and show that the conditions they impose on the line concur. The total number of lifts will be the product of the  \emph{local lifting multiplicities}, $m(C,\Lambda,P_i)$, at  each point $P_i$. The local lifting multiplicity for the cases (1), (2) and (4) can be found in Table \ref{table1}.  The lifting multiplicity for the case (3) can be found in Proposition \ref{prop-horizontalsegment} (see also Lemma \ref{nonRegularIntersection} and Table \ref{table1}), and for case (5) in Proposition \ref{prop-vertexofboth}.

\begin{theorem} \label{Thm:Lifting}
Let $\Trop(C)=\Gamma$ be generic in the sense of Assumption \ref{genass}, and let $(\Lambda,P_1,P_2)$ be a tropical bitangent, such that $P_1$ and $P_2$ are in different connected components of $\Gamma\cap\Lambda$.
If $P_1$ and $P_2$ are in the interior of the same edge of $\Lambda$, then $(\Lambda,P_1,P_2)$ doesn't lift to a bitangent to $C$.
Otherwise, there are $m(C,\Lambda,P_1)\cdot m(C,\Lambda,P_2)$ many algebraic bitangents of $C$ tropicalizing to $\Lambda$ such that the tangency points tropicalize to $P_1$ and $P_2$,  where $m(C,\Lambda,P_i)$ is specified in each case individually.

%
%

\end{theorem}

 For each of the cases (1)-(5), we assume without restriction that the tangency point $P$  is at $(0,0)$, which is on (the closure of) the horizontal end of $\Lambda$, and that the vertex of $\Lambda$ is at $(s,0)$ for some $s\in \mathbb{Z}_{\geq 0}$. Then the defining equation of any lift $\ell$ of $\Lambda$ is of the form 
\begin{equation}y+m+t^snx=0,\label{eq-line}\end{equation}
where $\val(m)=\val(n)=0$.

In order to state the genericity condition \ref{genass}, we first need to define an invariant of  tangency points which we call the \emph{initial lift}. 
Propositions \ref{prop:properGeneral} and \ref{prop-horizontalsegment} motivate our choice of terminology:  the initial lift equals the residue $\overline{m}$ of the constant coefficient of the line equation (\ref{eq-line}) in both cases (2) and (3).

\begin{definition}\label{def-initiallift}
Let $C=V(q)$ be a curve defined over $K$, and let $\Lambda$ be a bitangent of $\Gamma=\Trop(C)$, that falls into one of the cases (2) and (3) above. 
Then the  the \emph{initial lift} of $\Lambda$ and $C$ at $P$, denoted $\ini(\Lambda,C,P)$,  is as follows. 
\begin{itemize}
\item (Case 2) Let $(a+ bx^iy^j + cx^ky^l)\cdot x^{i'}y^{j'}$ be the three terms of  $q$ that correspond to the three regions of $\RR^2$ locally cut out around $P$ by $\Gamma$. 
Then the initial lift $\ini(\Lambda,C,P)$ is defined to be the residue of 
$-(-\frac{ka}{b(k-i)})^{-k} \cdot (-\frac{ia}{c(k-i)})^{i}$.
\item (Case 3) Let $a x^iy^j$ and $b x^ky^l$ be the terms of $q$ that correspond to the two regions of $\RR^2$ adjacent to $P$. 
Then the initial lift $\ini(\Lambda,C,P)$ is defined to be the residue of the quotient $\frac{a}{b}$.   
\end{itemize}
\end{definition}

The genericity assumption in Theorem \ref{Thm:Lifting} can then be specified as follows: 
\begin{assumption}\label{genass}
With notation as in Theorem \ref{Thm:Lifting}, we require that the tropicalization $\Gamma$ of $C$ is smooth. Furthermore, if $P_1$ and $P_2$ are in the interior of the same edge of $\Lambda$, we assume that the initial lifts do not agree, i.e.\ $\ini(\Lambda,C,P_1)\neq \ini(\Lambda,C,P_2)$. If one of the points falls into case (5), we impose an additional condition, which is specified before Proposition \ref{prop-vertexofboth}.
\end{assumption}
\noindent This is a condition on the coefficients of the defining equation $q$ of $C$, and there are only finitely many cases for possible $\Lambda$ and $P_i$ to check.

\begin{remark}
In Section 5 of \cite{Chan2}, Chan and Jiradilok compute initials of lifts of tangent line for a fixed quartic curve in Honeycomb form. Their quartic does not satisfy our genericity assumption \ref{genass}, since all Puiseux series coefficients start with a coefficient one. In particular, initial lifts for intersections contained in the same edge of $\Lambda$ coincide. As a result, some of the tropical lines in that example would not lift generically, e.g.\ the one with the very left vertex in Figure 5 of \cite{Chan2}.
\end{remark}

\begin{proof}[Proof of Theorem \ref{Thm:Lifting}]
If one of the tangency points is in the interior of an edge of both $\Lambda$ and $\Gamma$, then Proposition \ref{prop-transverse} shows that there is no bitangent lifting $\Lambda$. This takes care of case (1). Otherwise,  as described in Section \ref{sec-wronskian}, each tangency point $P_i$ determines $m(C,\Lambda,P_i)$ possible values for  one of the coefficients $m$ of the line, and doesn't impose any condition on the other coefficient. 

If the points are in the interior of the same end of $\Lambda$, then they each determine the same coefficient of the line. Together with the genericity assumption, we reach a contradiction, and it follows that the line doesn't lift. Otherwise, each point determines the value of a different coefficient,  and we obtain $m(C,\Lambda,P_1)\cdot m(C,\Lambda,P_2)$ global lifts of $\Lambda$. 
The different cases (2)--(5) are dealt with in Propositions,  \ref{prop:properGeneral}, \ref{prop-horizontalsegment}, \ref{prop-vertexoflambda}, \ref{prop-vertexofboth} below respectively. 

\end{proof}

We now go through the cases (1)-(5).

\begin{proposition}[Case (1)]\label{prop-transverse}
Suppose that $\Lambda$ and $\Gamma=\Trop(C)$ meet transversally at $P$ along the interior of edges of both of them (see Figure \ref{fig-alltypes}). Then there is no line $\ell$ that is tangent to $C$ at a lift of $P$. In particular, $m(C,\Lambda, P)=0$.
\end{proposition}

\begin{proof}
As in Section \ref{sec-wronskian}, we obtain three equations $q$, $\ell$ and the Wronskian $W$, which are listed in the second column of Table \ref{table1}. One easily checks that any solution would have $\val(x)=\val(y)=0$, but then $W$ has a unique term of lowest valuation. Hence, there is no solution.
%
%
\end{proof}

\begin{proposition}[Case (2)]
\label{prop:properGeneral}
Suppose that $\Lambda$ and $\Gamma=\Trop(C)$ (where $C$ satisfies the genericity assumption \ref{genass}) meet properly along a vertex $P=(0,0)$ of $\Gamma$ in the interior of the horizontal edge of $\Lambda$ (see Figure \ref{fig-alltypes}). Then for any Puiseux series $n$ of  valuation 0,  there exists a unique bitangent $\ell = y+m+t^kn$ of $C$ lifting $\Lambda$, where $k$ is the distance from $P$ to the vertex of $\Lambda$. In particular, $m(C,\Lambda,P)=1$, and the initial lift $\ini(\Lambda,C,P)$ as defined in \ref{def-initiallift} equals $\overline{m}$.
\end{proposition}

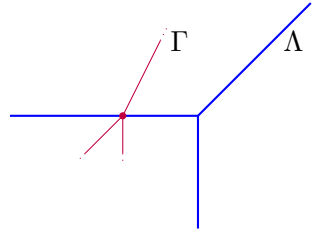
\begin{figure}
\centering
\begin{tikzpicture}[scale=.5]

\draw[thick, blue] (0,0) to (-5,0);
\draw[thick, blue] (0,0) to (0,-3);
\draw[thick, blue] (0,0) to (3,3);
\node [right] at (2,2) {$\Lambda$};

\draw[purple] (-2,0) to (-1,2);
\draw[purple, dotted] (-1,2) to (-0.8,2.4);
\draw[purple] (-2,0) to (-2,-1);
\draw[purple, dotted] (-2,-1) to (-2,-1.2);
\draw[purple] (-2,0) to (-3,-1);
\draw[purple, dotted] (-3,-1) to (-3.2,-1.2);
\draw[fill,purple] (-2,0) circle [radius=0.08];

\node [right] at (-1,2) {$\Gamma$};

\end{tikzpicture}   
\caption{Case (2): Proper intersection along a horizontal edge.}
\label{fig:proper}
\end{figure}

\begin{proof}
Again, as in Section \ref{sec-wronskian}, we obtain three equations, which are listed in the third column of Table \ref{table1}.
Note that $k\neq i$: otherwise, $\Gamma$ would have a horizontal edge adjacent to $P$, contradicting the assumption that the intersection, locally, consists of a single point. 
Since $\Gamma$ is smooth, the vectors $(i,j)$ and $(k,l)$ form a unimodular basis. We may therefore perform a change of coordinates $A = \overline x^i \overline y^j$ and $B = \overline x^k \overline y^{l}$, and solve for $A,B$. One easily checks that 
the values for $x,y,m$ appearing in the table are the unique solution for  the system of equations. 
It follows that $\overline m$ equals the initial lift $\ini(\Lambda,C,P)$, which by definition \ref{def-initiallift} is the residue of 
\[
-\Big(-\frac{k\overline{a}}{\overline{b}(k-i)}\Big)^{-k}\cdot \Big(-\frac{i\overline{a}}{\overline{c}(k-i)}\Big)^{i}.
\]

In order to make use of the non-Archimedean implicit function theorem as outlined in Section \ref{sec-wronskian}, we still need to show that the determinant of the residue of the Jacobian of the equations $\ell$, $q$ and $W$ with respect to the variables $x$, $y$, and $m$, after plugging in the initial values $\overline x$, $\overline y$ and $\overline m$, is nonzero. 
 We have $\overline{q_y(\overline x,\overline y)}$ is nonzero if and only if $j \overline b\overline x^i \overline y^{j-1}+ l\overline c \overline x^k \overline y^{l-1} $ is nonzero. The latter equals
$$ j\overline b\frac{A}{y}+l\overline c\frac{B}{y}= \frac{1}{y}\cdot \Big(-j\overline b \frac{\overline ckB}{i\overline b}+l\overline c B\Big)= \frac{\overline cB}{iy}\Big(-jk+il\Big)= \pm \frac{\overline cB}{iy} \neq 0,$$ where we use that the vectors $(i,j)$ and $(k,l)$ form a unimodular basis, and that $i\neq 0$ since $\Gamma$ has no horizontal edge at $P$. 
Similarly, $\overline{q_{xx}(\overline x,\overline y)}$ is nonzero if and only if $i(i-1)\overline b A+k(k-1)\overline cB$ is nonzero, which equals
$$-i(i-1)\overline b\frac{k\overline cB}{i\overline b}+k(k-1)\overline cB= k\overline cB (-(i-1)+(k-1))=k\overline cB(k-i) \neq 0,$$
where we use that $k\neq i$ and $k\neq 0$, since $\Gamma$ has no horizontal edge at $P$.

None of the computations above depended on the value of $n$, and therefore the tangent lifts uniquely for any choice of trivially valued $n$.

\end{proof}

 \begin{sidewaystable}
    \centering \small
{\renewcommand{\arraystretch}{0.35}}%
  \begin{tabular}{|c||c|c|c|c|}
    \hline
    \small{ }\normalsize& Case (1) (Prop \ref{prop-transverse}) 
    & Case (2) (Prop \ref{prop:properGeneral}) & Lemma \ref{nonRegularIntersection} & Case (4) (Prop \ref{prop-vertexoflambda})\\ 
    \hline
    \hline 
$q$ & $(ax^i + by^j)\cdot x^{i'} y^{j'}  $& $(a+ bx^iy^j + cx^ky^l)\cdot x^{i'}y^{j'}$  &  $a + cx^2 + dxy $& $(a\cdot x^\lambda+b\cdot y^\mu)\cdot x^i y^j$\\ \hline
with &  $i,j>0$ & $ij-kl=1$ & --- & --- \\ \hline
$\ell$ & $y+m$ &$y+m$ & $y+m$ & $y+m+nx$\\ \hline
\multirow{6}{*}{$W$} & \multirow{2}{*}{$iax^{i-1}\cdot x^{i'}y^{j'}$}&\multirow{2}{*}{$(ibx^{i-1}y^j + $}  & & \\ & \multirow{3}{*} {$  + (ax^i+by^j)$} &\multirow{3}{*}{$ckx^{k-1}y^l)x^{i'}y^{j'} + W_0$} & $2cx+dy$& $(a  \lambda x^{\lambda-1} - \mu b y^{\mu-1}\cdot n)\cdot x^i y^j$\\
 & & & &\\ 
&\multirow{2}{*}{$\cdot i'x^{i'-1} y^{j'}$} & & & \\
 & & & &\\ \hline
$J$ & --- & $
\begin{pmatrix}
  0 & 1 & 1 \\
  0 & \overline{q_y(\overline x,\overline y)} & 0 \\
  \overline{q_{xx}(\overline x,\overline y)} & \overline{q_{xy}(\overline x,\overline y)} & 0
 \end{pmatrix} $
 & $\begin{pmatrix}
  0 & 1 & 1 \\
  2\overline{c}\overline x + \overline{d}\overline{y} & k\overline{d}\overline{x} & 0 \\
  2\overline{c} & \overline{d} & 0
 \end{pmatrix}$& $\begin{pmatrix}
  \overline n & 1 & 1 \\
 \overline a \lambda \overline x^{\lambda-1} \overline x^i \overline y^j  & \overline b \mu \overline y^{\mu-1}\overline x^i \overline y^j & 0 \\
 \overline a \lambda(\lambda-1)\overline x^{\lambda-2} \overline x^i \overline y^j& -\overline b \mu(\mu-1)\overline y^{\mu-2} \overline n \overline x^i \overline y^j& 0
 \end{pmatrix} 
 $\\ \hline
$x$ & --- &$A^{l}B^{-j}$ & $\pm\sqrt{\frac{\overline{a}}{\overline{c}}}$& $[-\frac{\overline a}{\overline b}(-\frac{\lambda}{\mu \overline n}^\mu)]^{\frac{1}{\mu-\lambda}}$\\ \hline
$y$ & --- & $A^{-k}B^i$& $\mp\frac{2\sqrt{\overline{a}\overline{c}}}{\overline{d}} $& $-\frac{\mu \overline n}{\lambda}\overline x$\\ \hline
$m$ & --- &$-A^{-k}B^i$ & $ \pm\frac{2\sqrt{\overline{a}\overline{c}}}{\overline{d}}$& $\frac{\mu \overline n}{\lambda}\overline x-\overline n \overline x$\\ \hline
$m(C,\Lambda, P)$ & $0$ & $1$ & $2$& $|\det(u,v)|$\\ \hline

       \end{tabular}\bigskip
\caption{This table shows the leading terms of the equations (any further summand is $O(t)$). In each column, the valuations of the coefficients are zero: $\val(a)=\val(b)=\val(c)=0$. 
In the third column, $A$ and $B$ are expressions that simplify the computations: $A = \overline x^i \overline y^j=-\frac{\overline c k B}{i\overline b}$ and $B = \overline x^k \overline y^{l}=\frac{\overline a i}{\overline c(k-i)}$. $W_0$ is a term that vanishes on every solution of $q$. In the last column, $u=\binom{\mu}{\lambda}$ is the direction of the edge of $\Gamma$ passing through the vertex of $\Lambda$, and $v$ is the direction of the edge of $\Lambda$ which carries the second tangency point. }\label{table1}  \end{sidewaystable}

 \begin{proposition}[Case (3)]\label{prop-horizontalsegment}
 Suppose that $P$ is in the midpoint of a line segment in the intersection of $\Lambda$ and $\Gamma$ (see Figure \ref{fig-alltypes}), and  $C$ satisfies the genericity assumption \ref{genass}.
Then for any Puiseux series $n$ of valuation $0$ there are exactly two bitangent lines $\ell$ with equations $y+m+t^kn\cdot x$ lifting $\Lambda$, where $k$ is the distance from the segment to the vertex of $\Lambda$. In particular, $m(C,\Lambda,P)=2$. The initial lift $\ini(\Lambda,C,P)$ as defined in \ref{def-initiallift} 
equals $\overline{m}$ for each of the two local solutions.
 \end{proposition}
 
%

 \begin{proof}
As always, we assume without restriction that the segment is horizontal, and that $P=(0,0)$.
Since $\Gamma$ is smooth, $\Lambda$ and $\Gamma$ locally intersect with multiplicity 2. 

Let $e$ denote the edge of $\Trop(C)$ which intersects $\Lambda$ in a segment. Let $e^\vee$ denote the dual edge in the Newton subdivision. Without restriction, we can assume that the negative valuations of the coefficients of $q$ corresponding to the end vertices of $e^\vee$ are $0$. We let $i$ be the exponent of $x$ in the terms corresponding to those two end vertices, i.e.\ the edge $e^\vee$ is in the column $\{x=i\}$ in the Newton polygon.

We distinguish two cases, depending on whether the distance $k$ of the right adjacent vertex of $e$ from the vertex of $\Lambda$ (measured from the left) is positive or not. Both cases are depicted in Figure \ref{fig-alltypes}.
Assume first that $k>0$.
Because we assume smoothness of $\Gamma=\Trop(C)$, the edge $e^\vee$ must form triangles with two vertices on the left and right of lattice distance $1$ to $e^\vee$. Their negative valuations must be equal, and strictly smaller $0$, since we assume that $(0,0)$ is the midpoint of the edge. Assume their negative valuations are $-\lambda$ for some $\lambda>0$.

If $k>0$, we can assume that $n=0$, and $\Lambda$ is the horizontal line $y=0$. We can do this, since  $y+m+nt^kx$ is tangent if and only if $y+m$ is.

\begin{figure} [h]
\centering
\begin{tikzpicture}[scale=.5]
\footnotesize
\draw[thick, blue] (-2,0) to (2,0);
\draw[thick, blue] (-2,0) to (-1,1);
\draw[thick, blue] (-2,0) to (-4,-1);
\draw[thick, blue] (2,0) to (3,1);
\draw[thick, blue] (2,0) to (2,-1);

\draw[fill,blue] (0,0) circle [radius=0.08];
\node [above] at (0,0) {$(0,0)$};

\draw[fill,blue] (-2,0) circle [radius=0.08];
\node [left] at (-2,0) {$(-\lambda,0)$};

\draw[fill,blue] (2,0) circle [radius=0.08];
\node [right] at (2,0) {$(\lambda,0)$};

\node [left] at (-1.2,1) {$ q$};

\draw[step=1cm,gray,very thin] (-9.9,-1.9) grid (-5.1,2.9);
\node [left] at (-10,0) {$j$};
\node [below] at (-7,-2) {$i$};

\draw[fill,blue] (-7,0) circle [radius=0.08];
\draw[fill,blue] (-7,1) circle [radius=0.08];
\draw[fill,blue] (-8,2) circle [radius=0.08];
\draw[fill,blue] (-6,0) circle [radius=0.08];

\draw[blue] (-7,0) to (-7,1);
\draw[blue] (-7,1) to (-8,2);
\draw[blue] (-7,0) to (-8,2);
\draw[blue] (-7,0) to (-6,0);
\draw[blue] (-7,1) to (-6,0);

\node [below right] at (-6,0) {$-\lambda$};
\node [above left] at (-8,2) {$-\lambda$};

\end{tikzpicture}   
\caption{The curve and its dual polygon.}
\label{HorizontalAndNewton}
\end{figure}
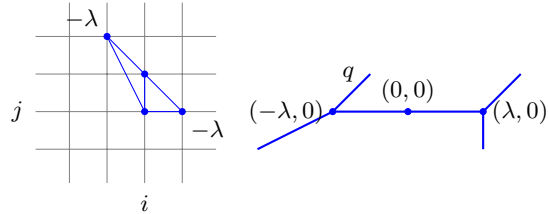
The coefficients of all the monomials corresponding to the  $x^{i-1}$ and $x^{i+1}$ column have negative valuation strictly smaller than $-\lambda$. Otherwise we would not get a triangle in the Newton polygon. Let 
\begin{equation}x^iy^j(y+a) \label{eq-initialmon}\end{equation} be the monomials of $q$ dual to $e$, where $a\in K$ has valuation $0$. 
Write 
\[
a = a_0 + a_1t +\ldots.
\]

Following the technique described in Subsection \ref{subsec-modification} (Case (I)), we modify the plane at $\Lambda$ and pick a re-embedding of $\Trop(C)$ that will produce a proper intersection. We set $z = y+a_0$ and consider the projection of the re-embedded curve in the modified plane, defined by $\tilde{q} = q(x,z-a_0)$. A line $\ell(x,z-a_0)$ is bitangent to $\tilde{q}$ if and only if $\ell$ is bitangent to $q$, as in Remark \ref{rem-coordinatechange}.
To compute the Newton subdivision of $\tilde{q}$, note that the term $x^iy^{j+1}$ and $a\cdot x^i y^j$  of $q$ are replaced by $x^i(z-a_0)^{j+1}$ and $a\cdot x^i(z-a_0)^j$ respectively. 
Then 
\[
(z-a_0)^{j+1} + (z-a_0)^j a = z^{j+1} + \sum_{k=0}^{j}\bigg[\binom{j+1}{k}(-a_0)^{j+1-k} + \binom{j}{k}a(-a_0)^{j-k}\bigg]\cdot z^k.
\]
It follows that the coefficient of $x^i\cdot z^0$ is $(-a_0)^{j+1} + (-a_0)^j\cdot a$ which has valuation strictly larger than $0$. All the other coefficients on the  $x^i$ column have valuation $0$.

On the $x^{i-1}$  column, the point corresponding to $x^{i-1}z^j$ and all the points below it  have negative valuation $-\lambda$. This is because $x^{i-1}y^{j}$ was the only point at that height in the Newton polygon of $q$ (as a consequence of our smoothness assumption on $\Gamma$) and so no cancellation can happen.
The same holds for the coefficients in the $x^{i+1}$-column.
It follows that Newton polygon of $\tilde{q}$ locally looks as in Figure \ref{NewtonAfterModification}.
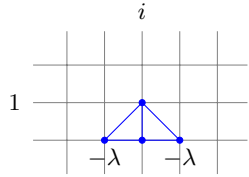
\begin{figure} [h]
\centering
\begin{tikzpicture}[scale=.5]
\footnotesize
%
%
%
%

\draw[step=1cm,gray,very thin] (-9.9,-2.9) grid (-4.2,.9);
\node [left] at (-10,-1) {$1$};
\node [above] at (-7,1) {$i$};

\draw[fill,blue] (-7,-1) circle [radius=0.08];
\draw[fill,blue] (-6,-2) circle [radius=0.08];
\draw[fill,blue] (-8,-2) circle [radius=0.08];
\draw[fill,blue] (-7,-2) circle [radius=0.08];

\draw[blue] (-8,-2) to (-7,-1);
\draw[blue] (-7,-1) to (-6,-2);
\draw[blue] (-6,-2) to (-8,-2);
\draw[blue] (-7,-2) to (-7,-1);

\node [below ] at (-6,-2) {$-\lambda$};
\node [below] at (-8,-2) {$-\lambda$};

\end{tikzpicture}
\caption{The Newton polygon after the first modification.}
\label{NewtonAfterModification}
\end{figure}

The point corresponding to the $x^i z^0$ is visible in the Newton subdivision only if it is above the line between the points corresponding to $x^{i-1} z^0$ and $x^{i+1} z^0$.

We then adapt our modification such that more and more coefficients get cancelled. The exact numbers we have to pick for this depend on coefficients of higher terms of $a$ and on other coefficients of $q$, but in any case, we can find a modification of the form $z = y+a_0 + \ldots$ such that the Newton subdivision and the curve are as in Figure \ref{NewtonFinal}. In particular, we can reach that the coefficient $x^i z^0$ has a negative valuation that is below the line imposed by the terms $x^{i-1} z^0$ and $x^{i+1} z^0$.
\begin{figure} [h]
\centering
\begin{tikzpicture}[scale=.5]
\footnotesize
\draw[blue] (-2,0) to (-1,-1);
\draw[blue] (-1,-1) to (0,0);
\draw[blue, thick] (-1,-1) to (-1,-2);
\node [left] at (-1,-1.5) {$2$};

\draw[step=1cm,gray,very thin] (-9.9,-2.9) grid (-4.2,.9);
\node [left] at (-10,-1) {$1$};
\node [above] at (-7,1) {$i$};

\draw[fill,blue] (-7,-1) circle [radius=0.08];
\draw[fill,blue] (-6,-2) circle [radius=0.08];
\draw[fill,blue] (-8,-2) circle [radius=0.08];

\draw[blue] (-8,-2) to (-7,-1);
\draw[blue] (-7,-1) to (-6,-2);
\draw[blue] (-6,-2) to (-8,-2);
\node [below ] at (-6,-2) {$-\lambda$};
\node [below] at (-8,-2) {$-\lambda$};

\end{tikzpicture}
\caption{The Newton polygon and dual curve after multiple modifications.}
\label{NewtonFinal}
\end{figure}
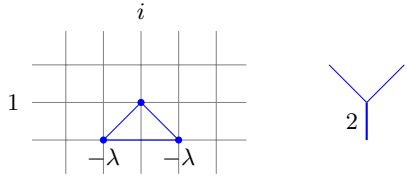

Depending on $m$, the re-embedded tropical line meets the tropicalization of the re-embedded $C$ either below, above, or at the vertex. Lemma \ref{nonRegularIntersection} below shows that lines meeting the curve away from the vertex do not lift to a tangent, and that the line that meets the curve at the vertex yields precisely two possible lifts $\tilde m$ for its constant coefficient. 

Since the re-embedded lift $\ell$ of $\Lambda$ has to show the special cancellation behaviour (since generically we obtain the modification of $\Lambda$ which does not meet our requirements above), we conclude that for each of the two lifts, the initial coefficient of $m$ equals $a_0$, which by equation (\ref{eq-initialmon}) equals the residue class of the quotient of the two coefficients of the defining equation $q$ of $C$ corresponding to regions of $\RR^2\setminus \Gamma$. The latter equals the initial lift  $\ini(\Lambda,C,P)$ by definition \ref{def-initiallift}.
Lemma \ref{nonRegularIntersection} provides two distinct solutions for the coefficient of $m$ of lowest order which is not determined by our choice of modification. 

If $k\leq 0$, the argument works similarly. We have to modify at $\Lambda$ defined by the tropical polynomial $\max\{0,N+x,y\}$  with $N< 0$. We are in case (II) described in Subsection \ref{subsec-modification}, and so in principle we have to consider three projections and the cone feeding. However, since the coefficient $n$ of $x$ in a lift $\ell$ of $\Lambda$ has a higher valuation, we can neglect its effect and focus on the ``vertical down'' feeding in the projection to $x$ and $z$ as in Figure \ref{fig-modiNewt1}. Accordingly, we consider only the part of the curve in the cell $\sigma_4$ of the modified plane, because this is the cell that contains the point of tangency. For that reason, we can follow along the same lines as in the case where $k>0$ and obtain a re-embedding such that $\Trop(\tilde{g})$ has a Newton subdivision as depicted in Figure \ref{NewtonFinal} (except $(0,0)$ is not the midpoint of the edge $e$, but the midpoint of the segment of intersection of $\Trop(C)$ and $\Lambda$). Applying Lemma \ref{nonRegularIntersection} again yields the result. 
\end{proof}

\begin{remark}\label{rem-nonproper3valent}
When $k=0$ of Proposition \ref{prop-horizontalsegment} requires more attention: in that case, the vertex of the line coincides with a vertex of the curve, and the stable intersection multiplicity along the component of intersection may be higher than $2$. Therefore,  we may have more than one tropical tangency point in the component. Under the assumption in this subsection, that the points $P$ and $Q$ of $(\Lambda,P,Q)$ are in different connected components, the same methods as in Proposition \ref{prop-horizontalsegment} apply, and the local lifting multiplicity is still $2$. When $P$ and $Q$ are in the same component, this is dealt with in Subsection \ref{sec-starshaped}.

\end{remark}

The following lemma was used in the proof of Proposition \ref{prop-horizontalsegment} above, for a point $P$ in the newly attached cell of the modified plane. Locally, the point had coordinates $P=(P_x,0)$ with $P_x<0$. 
In the proof below, we assume without restriction that the point of tangency is at the origin --- this just amounts to a shift, and  multiplication with an appropriate power of $t$.

\begin{lemma}\label{nonRegularIntersection}
Denote $v=(0,0)$, and let $\Gamma=\Trop(\tilde{q})$  with
\[
\tilde{q} = a+bx+cx^2+dxy+O(t), \val{a}=\val{c}=\val{d} < \val{b}.
\] 
Let $\Lambda$ be a tropical line with vertex at $(k,l)$ for some $k>0$,  whose horizontal end meets $\Gamma$. If the intersection is at the vertex $v$, then for any choice of trivially valued $n$, there are two tangents $\ell=y+\tilde m+t^kn$ of $V(\tilde q)$ lifting $\Lambda$. Otherwise, $\Lambda$ doesn't lift. 
%
\end{lemma}

\begin{proof}
As usual, we set up the three equations as in Subsection \ref{sec-wronskian} (see Table \ref{table1}). When the line passes below the vertex, the assumption on the valuation of the coefficients implies that there is no solution to the equations. When the line passes at the vertex, one checks that, for any value of $n$, there are two solutions for the initial terms of $x,y,m$, such that the conditions of Theorem \ref{Thm:IFT} are satisfied.

\begin{figure}[h]\label{fig:nonSmooth}
\centering
\begin{tikzpicture}[scale=.5]

\draw[thick, blue] (0,0) to (-5,0);
\draw[thick, blue] (0,0) to (0,-3);
\draw[thick, blue] (0,0) to (3,3);
\node [right] at (2,2) {$\ell$};

\draw[fill,purple] (-2,0) circle [radius=0.08];
\draw[purple] (-2,0) to (-3,1);
\draw[purple, dotted] (-3,1) to (-3.2,1.2);
\draw[purple] (-2,0) to (-1,1);
\draw[purple, dotted] (-1,1) to (-0.8,1.1);
\draw[purple] (-2,0) to (-2,-2);
\draw[purple, dotted] (-2,-2) to (-2,-2.2);

\node [right] at (-1,2) {$\tilde q$};
\node [right, font=\footnotesize] at (-2,-1) {$2$};

\end{tikzpicture}   
\caption{A local intersection as in Lemma \ref{nonRegularIntersection} --- in particular, the type of intersection we obtain after modifying.}
\label{CurveAfterModification}
\end{figure}
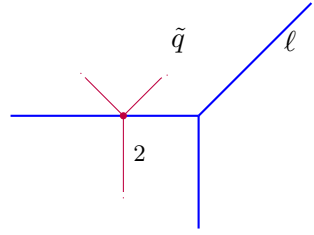

\end{proof}

\begin{proposition}[Case (4)]\label{prop-vertexoflambda}
Suppose that $\Lambda$ and $\Gamma=\Trop(C)$ (where $C$ satisfies the genericity assumption \ref{genass}) meet with multiplicity greater one at the vertex $P=(0,0)$ of $\Lambda$, which is in the interior of an edge $e$ of direction $u$ of $\Gamma$ (see Figure \ref{fig-alltypes}). Assume the other tangency point $Q$ is on the end of $\Lambda$ of direction $v$. Then there are $m(C,\Lambda,Q)\cdot |\det(u,v)|$ bitangents $\ell = y+m+nx$ lifting $\Lambda$.
\end{proposition}


\begin{proof}
Let $u=\binom{\mu}{\lambda}$. We assume without restriction that the other tangency point $Q$ is on the diagonal end of $\Lambda$, so that $|\det(u,v)|=|\mu-\lambda|$. The other cases follow from this case using the appropriate coordinate change.
Since the intersection multiplicity is at least $2$, it follows that $\lambda,\mu\neq 0$, and that $\lambda\neq \mu$.  
Table \ref{table1} lists the three equations as in Subsection \ref{sec-wronskian} (where we use that the lowest order terms of $q$ vanish in $(x,y)$).
The fact that $Q$ is on the diagonal end of $\Lambda$ implies that we can view one of $m(C,\Lambda,Q)$ possible lifts for $n$ as fixed, and we have to find lifts for $m$.

To compute the Jacobian (see table \ref{table1}), we use again that the lowest order terms of $q$ vanishes at $(x,y)$.
Its determinant is 
\begin{align*} &-\overline x^{2i} \overline y^{2j} a b  \lambda \mu \overline x^{\lambda-2}\overline y^{\mu-2}[(\mu-1)\overline n \overline x + (\lambda-1)\overline y]\\ = & 
\overline x^{2i} \overline y^{2j} a b   \mu \overline x^{\lambda-2}\overline y^{\mu-2}
\overline n \overline x  [ (\mu-1) \lambda  - \mu (\lambda -1)]\\ =&  
\overline x^{2i} \overline y^{2j} a b   \mu \overline x^{\lambda-2}\overline y^{\mu-2}
\overline n \overline x  [  \mu-\lambda],
\end{align*}
 where in the first equality we use $\overline y = -\frac{\mu \overline n}{\lambda}\overline x$. This is nonzero for any choice of initial $\overline x$, $\overline y$, since $\lambda\neq \mu$.
\end{proof}

We next deal with the case where the tropical tangency point is at a vertex $P$ of both the line and the curve. 
We set up notations as follows.   
Assume (without restriction) that the triangle dual to the vertex of $\Gamma$ has vertices $(i,0), (0,j), (k,l)$ such that $j>i,0$, and $(k,l)$ is  to the right of the affine line spanned by the  two other vertices (see Figure \ref{fig-dualTriangle}). Further, we assume that the second tangency point $Q$ is in the interior of the diagonal end of the line, so that $n$ is already fixed. Note that these conditions do not restrict generality, since they will always be satisfied after multiplying by a suitable power of $y$ or permuting the parameters.
The curve will further need to be generic in the sense that it doesn't have any hyperflex points or tritangent lines (these are indeed generic conditions, see \cite[Proposition 11.10]{EH} and \cite{MO}), and $n$ avoids the following values: $\frac{a}{b}, \frac{ac^{j-1}}{(-b)^j},  \frac{cb^{-l-1}}{(-a)^{-l}}, \frac{bc^{i-1}}{(-a)^i}$.

\begin{figure}
\begin{tikzpicture}[scale=.7]
  
\coordinate (x) at (0,4); 
\coordinate (y) at (2,0); 
\coordinate (z) at (3,3); 

\draw [blue] (x)--(y)--(z)--cycle;
\node [left] at (x) {$(0,j)$};
\node [below] at (y) {$(i,0)$};
\node [right] at (z) {$(k,l)$};
\end{tikzpicture} 
\caption{The triangle $\Delta_\Gamma$ dual to a vertex of $\Gamma$.} \label{fig-dualTriangle}

 \end{figure}
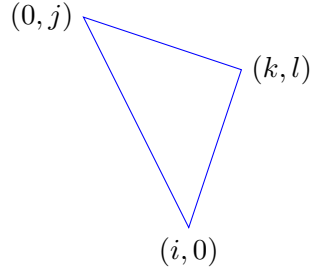

\begin{proposition}[Case (5)]\label{prop-vertexofboth}
With notations as above, when a curve $\Gamma$ and a line $\Lambda$ meet at a vertex $P$ of both, there are $m(C,\Lambda,Q)\cdot M$ bitangents $\ell = y+m+nx$ lifting $\Lambda$, where $M$ is determined as follows.

If $(i,j)\neq(1,1)$ and $i,l\neq 0$ then 
 \[
    M  = \begin{cases}
k+l-i+1 & \text{if $i+1<j+1\leq k+l$}\\
j-i+1 & \text{if $i+1\leq k+l<j+1$}\\
j+1-k-l & \text{if $k+l<i+1<j+1$}.\\
    
                    \end{cases}  
\]

If $l=0$ and $i\neq 1$ then $i$ is non-positive and $M=1-i$.

If $i=0$ then $M=l+1$  when $l$ is positive,  and $M=-l$ when $l$ is negative. 

If $i=j=1$ then
\[
M = \begin{cases}
1 & \text{if $k$ or $l$ equals $0$}\\
2 & \text{if $k,l\neq 0$}.\\
\end{cases}
\]


 \end{proposition}
\begin{proof}
The condition that $(k,l)$ is to the right of the other edge, together with the regularity of $\Gamma$ is equivalent to the identity $il+jk-ik=1$. We assume without loss of generality that $P$ is at the point $(0,0)$, and as always, we obtain $3$ equations, 
\begin{align*}
&\ell = y+m+nx,\\
& q=ax^i+by^j+cx^ky^l + O(t),\\ 
&W = aix^{i-1}+ckx^{k-1}y^l - n(bjy^{j-1} + clx^ky^{l-1}) + O(t)
\end{align*}
in variables $x$, $y$ and $m$. The number of bitangents tropicalizing to $\Lambda$ equals the number of intersection points of $q,\ell$ and $W$ with trivial valuation. Since $m$ only appears in the equation $\ell$, and we assume that $n$ is fixed by the tangency point $Q$, the number of solutions equals the intersection multiplicity of $q$ and $W$ (note that by the assumption that there are no tritangent lines, different choices of the point of tangency imply that $m$ is different as well).

We compute the intersection multiplicity using tropical intersection theory: if  $\Gamma$ and $\Trop(V(W))$ intersect transversally,  then  by \cite{OP}, the number of lifts equals the local tropical intersection number at the vertex. If they don't intersect transversally, we replace them with an equivalent system of equations that does. Denote $\Omega=\Trop(W)$, and $\Delta_\Gamma,\Delta_\Omega$ the polygons dual to the vertex at the origin of $\Gamma$ and $\Omega$ respectively.

We begin with the case $i=j=1$. Then $W=a-nb+kcx^{k-1}y^l-nlcx^ky^{l-1}$, and the constant term is non-zero by the assumption $n\neq \frac{a}{b}$. By  regularity of $\Gamma$, it follows that $k+l=2$. If neither $k$ or $l$ equals $0$,  then the rays of $\Gamma$ and $\Omega$ are
\[
\binom{1-l}{k},\binom{l}{1-k},\binom{-1}{-1}
\]
and
\[
-\binom{1-l}{k},-\binom{l}{1-k},-\binom{-1}{-1},
\]
respectively, and they intersect properly with multiplicity $2$ (see Figure \ref{fig-vertexvertex-ab1} on the left).  If either $k$ or $l$ is $0$, it is straightforward to check that there is a unique common solution for $q$ and $W$ with trivial valuation.

\begin{figure}
\begin{tikzpicture}[scale=.5]

\draw[purple] (0,0) to (-3,-3);
\draw[purple] (0,0) to (-3/2,-3);
\draw[purple] (0,0) to (2,3);
\node [left] at (-2,-1.9) {$\Gamma$};

\draw[teal] (0,0) to (3,3);
\draw[teal] (0,0) to (3/2,3);
\draw[teal] (0,0) to (-2,-3);
\node [right] at (2,2)  {$\Omega$};

  \begin{scope}[shift={(9,0)}]

\draw[purple] (0,0) to (-4,-1);
\draw[purple] (0,0) to (-3,-1);
\draw[purple] (0,0) to (3.5,1);
\node [right] at (-4,-0.25) {$\Gamma$};

\draw[teal] (0,0) to (1,1);
\draw[teal] (0,0) to (4,1);
\draw[teal] (0,0) to (2,0);
\draw[teal] (0,0) to (-3.5,-1);
\node [left] at (2,1.4) {$\Omega$};

\end{scope}

\end{tikzpicture} 
\caption{$\Trop(q)$ and $\Trop(W)$ for two different curves. On the left, $i=j=1,k=3,l=-1$, and on the right,  $i=1,j=3,k=-1,l=7$.} \label{fig-vertexvertex-ab1}

 \end{figure}
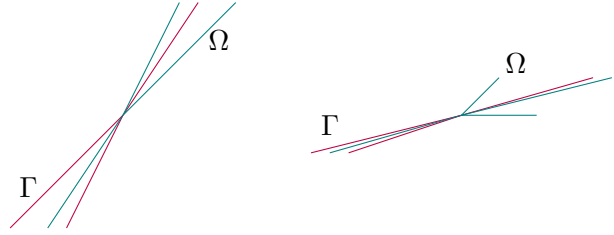

Next, assume that $(i,j)\neq(1,1)$, and none of $i,j,k,l$ is zero.  The polygon  $\Delta_\Omega$ has vertices $(i-1,0), (0,j-1), (k-1,l), (k,l-1)$. The value of $k+l$ determines the relative position of the vertices of $\Omega$: the vertex $(k,l-1)$ is to the right of the line spanned by $(i-1,0)$ and $(0,j-1)$ if and only if $k+l<j+1$, and similarly, $(k-1,l)$ is to the right of the line if and only if $k+l<i+1$. We treat each case individually.

Assume first that  $i+1<j+1<k+l$ (see Figure \ref{fig-vertexvertex-ab1}). 
When $i>0$, regularity implies that $k<0$ and  $l>j$. The condition $k+l>j+1$ implies that the line spanned between $(k-1,l)$ and $(k,l-1)$ hits the $y$ axis above $j-1$. It follows that one of the rays of $\Gamma$ will be in the $(1,1)$ direction. When  $i<0$,
regularity implies $k>0, l>j$, and we similarly obtain a ray in the $(1,1)$ direction. 
In any case, the rays of $\Gamma$ and $\Omega$, ordered according to their direction counterclockwise, are
\begin{center}
\begin{tabular}{c c c c c c c c}
 $\Gamma$: & & $\binom{j-l}{k}$&& $\binom{-j}{-i}$&&&$\binom{l}{i-k}$ \\
 $\Omega$: & $\binom{1}{1}$& &$\binom{-l}{k-i}$&& $\binom{j-1}{i-1}$ $\binom{l-j}{-k}. $\\  
\end{tabular}
\end{center}
Locally at the vertex, $\Gamma$ and $\Omega$ intersect properly with multiplicity $l+k-i+1$. 

If $k+l=j+1$, 
then the rays in counter clockwise order are
\begin{center}
\begin{tabular}{ c c c c c c c}
$\Gamma$: & $\binom{d}{i-k}$ & & $\binom{j-l}{k}$ & & $\binom{-j}{-i}$ \\ 
 $\Omega$: && $\binom{1}{1}$ & & $\binom{-l}{k-i}$ & & $\binom{l-1}{i-k-1}.$ \\  
\end{tabular}
\end{center}
Again, $\Gamma$ and $\Omega$ intersect transversally with multiplicity $j-i+2$.

%
%
%
%
%
%
%

If $i+1< k+l < j+1$ 
then the tropicalizations $\Gamma$ and $\Omega$ overlap along the ray $\binom{j-l}{k}$. From the equation for $q$, we have $cx^ky^{l-1}=\frac{-ax^i-by^j}{y} + O(t)$. We plug that into $W$ and multiply by $y$, to obtain a new equation
\[
W' = aiyx^{i-1} + ckx^{k-1}y^{l+1} +nby^j(l-j) + nalx^i + O(t).
\]
Note $l-j\neq0$, since otherwise $j=k=l=1$, contradicting the fact that $k+l<j+1$.
Now, $\Gamma$ and $\Omega'$ intersect properly, and their rays in counter clockwise order are
\begin{center}
\begin{tabular}{ c c c c c c c c}
$\Gamma$: & & $\binom{j-l}{k}$&&&$\binom{-j}{-i}$&&$\binom{l}{i-k}$\ \\ 
 $\Omega'$: & $\binom{j}{i}$& &$\binom{l-j+1}{1-k}$ &$ \binom{-l}{k-i}$&& $\binom{-1}{-1}. $\\  
\end{tabular}
\end{center}
The curves intersect properly with local multiplicity $j-i+1$ at the vertex, as claimed. 

If $i+1=k+l<j+1$, then 
$\Gamma$ and $\Omega$ overlap along the ray $\binom{j-l}{k}$.
As usual, we plug in $cx^ky^{l-1}=\frac{-ax^i-by^j}{y}+O(t)$ to $W$ and multiply by $y$ to get  $W'$, whose tropicalization $\Omega'$ has rays
\[
\binom{j}{i},\binom{-1}{-1},\binom{1-j}{1-i}.
\]
Now $\Omega'$ and $\Gamma$ intersect properly with multiplicity $j-i+1$. 

If $k+l<i+1<j+1$, 
we make the change of variables
 \[
 i'=k-i, j'=-l, k'=-i, l'=j-l,
 \]

\noindent which brings us to the case $j'+1\leq k'+l'$. As we have already seen,  the number of lifts is $l'+k'-i'+1 = j-l-i-k+i+1 = j-k-l+1$.

\begin{figure}
\begin{tikzpicture}[scale=.4]
\begin{scope}[shift={(-8,0)}]
\draw[teal] (0,0) to (4,2);
\draw[teal] (0,0) to (6,1);
\draw[teal] (0,0) to (-2,0);
\draw[teal] (0,0) to (-5,-2);
\node [left] at (-2,-0) {$\Omega$};
\end{scope}

\draw[purple] (0,0) to (4,2);
\draw[purple] (0,0) to (5,2);
\draw[purple] (0,0) to (-7/2,-3/2);
\node [right] at (2,2) {$\Gamma$};

\draw[teal] (0,0) to (7/2,3/2);
\draw[teal] (0,0) to (-2,-2);
\draw[teal] (0,0) to (-5,-2);
\draw[teal] (0,0) to (-2,0);
\node [left] at (-2,-0) {$\Omega'$};

\end{tikzpicture} 
\caption{$\Trop(W)$ on the left, and the intersection of $\Trop(q)$ and $\Trop(W')$ on the right, when $i=3,j=7,k=1,l=5$.} \label{fig-vertexvertex}

 \end{figure}
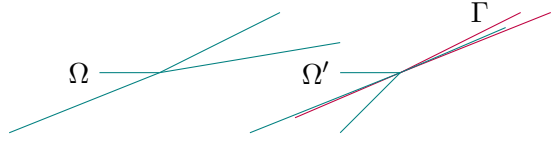

Finally, we deal with the case where at least one of $i,k,l$ is $0$. If $k=0$ then either $i=1, l=j+1$, and $j>1$; or $i=-1$ and $l=j-1\geq 0$. When
 $i=1$,  we have $q=ax+by^j+cy^{j+1}$ and $W=a-n(jby^{j-1}+(j+1)cy^j)$. In this case, the rays of $\Gamma$ around the origin are
\[
\binom{j-1}{1}, \binom{-1}{0}, \binom{-j}{-1},
\]
whereas $\Omega$ is just a horizontal line with multiplicity $j$,
and they overlap along the $\binom{-1}{0}$ direction. The stable intersection at the vertex  has multiplicity $j$. 
By \cite{OR}, there are $j$ intersection points that tropicalize to the segment where $\Omega$ and $\Gamma$ overlap. The $x$ coordinate of any solution that doesn't tropicalize to the vertex will have positive valuation. However, using the fact that $n\neq \frac{ac^{j-1}}{(-b)^j}$,  there are no such solutions, so the number lifts equals the stable intersection which is $j$. If $i=-1$,  then we similarly get $j+2$ solution. 
 
 If $i=0$, then $j=k=1$, and we distinguish between positive and negative $l$ (when $l=0$, a tropical line is not tangent at the point). If $l>0$, then $\Gamma$ and $\Omega$ intersect non-transversally. Once again, there are no common solutions when the valuation of $x$ is positive  (using $n\neq\frac{b^{l+1}}{c(-a)^l}$), and the number of lifts equals the multiplicity of the stable intersection $l+1$.
When $l<0$, we multiply everything by $y^{-l}$ (without changing the number of bitangents) and get 
 \[
 q' = ay{-l} + by^{-l+1} + cy.
 \]
By renaming  $i'=1, j'=-l, k'=0, l'=-l+1, a'=c,b'=a,c'=b$ we are in the  $k'=0, i'=1$ case, so there are $j'= -l$ lifts (using $n\neq \frac{a'c'^{j'-1}}{(-b')^{j'}} = \frac{cb^{-l-1}}{(-a)^{-l}}$).

If $l=0$, then $j=1$ and $k=i+1$. The only relevant case is where $i<0$ (otherwise $i=1$ which we already dealt with, or $i=0$ which means that a tropical line is not tangent at the point), which similarly to the other cases yields  $1-i$ lifts (using that assumption  $n\neq\frac{bc^{i-1}}{(-a)^i}$).

\end{proof}

\subsection{Lifting tropical lines tangent at two points of the same connected component} \label{sec-starshaped}
We now consider bitangent lines $(\Lambda,P,Q)$ such that  $P$ and $Q$ are in the same connected component of the intersection of $\Gamma$ with $\Lambda$. We still assume that $\Gamma$ is smooth and $P\neq Q$.
Since $P$ and $Q$ cannot lie in the interior of the same edge of the line, the intersection must contain a trivalent vertex $V$ of the curve, which coincides with the vertex of the line.
There are two cases to deal with. The first is when all the edges emanating from $V$ are aligned with the edges of the line. We refer to this case as ``star-shaped'' intersection. If the three edges of $\Gamma$ adjacent to $V$ are bounded, and the shortest of their length is unique, the bitangent lifts $4$ times (see Figure \ref{fig-starshaped} and Proposition \ref{prop-starshaped}). The other case is where the overlap is along a line segment, and one of the points of tangency is at the vertex. The bitangent doesn't lift in this case (see Figure \ref{fig-noliftforpointsinsegment} and Lemma \ref{lem-noliftforpointsinsegment}). 

For this subsection as before, we let $\Gamma=\Trop(C)$ and assume that $\Gamma$ is smooth.

%

\begin{figure}
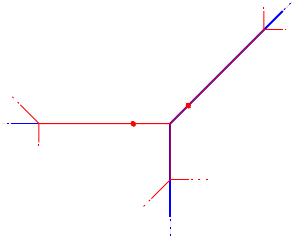
\caption{``Star-shaped'' intersection.}
\label{fig-starshaped}
\end{figure}

\begin{proposition} \label{prop-starshaped} Let $(\Lambda,P,Q)$ be a tropical bitangent such that $P$ and $Q$ are contained in a star-shaped intersection between $\Lambda$ and $\Gamma$. Assume that the three edges 
of $\Gamma$ adjacent to the vertex $V$ are bounded, 
and that $\Gamma$ is generic in the sense that the shortest length of the three adjacent edges is unique.
Then $\Gamma$ lifts to four bitangents of $C$.
\end{proposition}

\begin{proof}
Assume that the three edges of $\Gamma$ adjacent to $V$ have lengths $\lambda_1,\lambda_2$ and $\lambda_3$ with $\lambda_1< \lambda_2\leq \lambda_3$. Without restriction, we let $V$ be at the point $(0,0)$.

In the dual Newton subdivision of $\Gamma$, there is a triangle $V^\vee$ dual to $V$ with vertices corresponding to the monomials $x^iy^{j+1}$, $x^iy^j$ and $x^{i+1}y^j$. Assume that the defining equation $q$ of $C$ has coefficients $1$, $b$ and $a$ for these monomials. We can assume without restriction that the valuation of these three coefficients is $0$.
The three facets of the triangle $V^\vee$ each form another triangle with vertices that must be of lattice distance $1$ from the facets by our smoothness assumption, and the assumption that the intersection is bounded.
Without restriction --- possibly after a linear change of coordinates --- we can assume that the vertex $P_1$ corresponding to the term with coefficient of valuation $\lambda_1$ forms a triangle with the horizontal facet, the term with $\lambda_2$ belongs to the vertical facet and $\lambda_3$ to the diagonal (see Figure \ref{fig-Newtstarshaped}).
By our assumption $\lambda_1 < \lambda_2\leq \lambda_3$, there are exactly three possible positions for $P_1$, as depicted in Figure \ref{fig-Newtstarshaped}. If  $P_1$ was on the left of case (1), it would form a triangle with the vertical facet instead of the point corresponding to the term whose coefficient has valuation $\lambda_2$. If it was on the right of case (3), it would win over $\lambda_3$. Figure \ref{fig-Newtstarshaped} also depicts the possible locations for the vertices corresponding to $\lambda_2$ and $\lambda_3$.

\begin{figure}
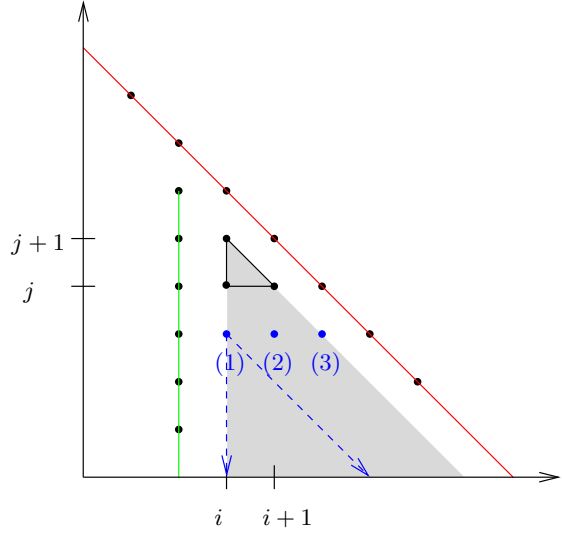
\caption{The Newton subdivision of $\Gamma$ near the vertex $V$. The black triangle is $V^\vee$ dual to $V$. The blue points are possible locations of $P_1$. The point corresponding to $\lambda_2$ can be any lattice point on the green line segment (if $\lambda_2<\lambda_3$, in case of equality the line segment can be extended to reach  the red line), the point $\lambda_3$ can be any point on the red line segment. The picture also sketches the effect of the ``cone feeding'' for the projection $q_{xz}$.}\label{fig-Newtstarshaped}
\end{figure}

We modify the plane at $\Lambda$ as in case (II) of Subsection \ref{subsec-modification}, via a linear re-embedding of $\Trop(C)$ using the lift $\alpha x +y+\beta$ of $\Lambda$. We study the linear re-embedding using the projection $\pi_{xz}$ given by $q_{xz}=q(x,z-\alpha x -\beta)$. For now, consider $\alpha$ and $\beta$ as variables standing for Puiseux series of valuation $0$. We will specify how to pick them in the three different cases corresponding to the locations of $P_1$ as we proceed.

We start by computing the contributions of the three vertices of $V^\vee$ to the monomials $x^{i+k}$ for $0\geq k\geq j+1$  and to $x^{i+k}z$ for $k=0$ and $k=j$ in $q_{xz}$.
That is, we need to compute the coefficients of those monomials in
$$x^i(z-\alpha x -\beta)^{j+1}+b\cdot x^i (z-\alpha x -\beta)^j+a\cdot x^{i+1}(z-\alpha x -\beta)^j.$$

For $x^{i+k}$, we obtain

\begin{align*}
&b\cdot \binom{j}{k}(-1)^j\alpha^k\beta^{j-k}+\binom{j+1}{k}(-1)^{j+1}\alpha^k\beta^{j+1-k}+a\cdot \binom{j}{k-1}(-1)^j\alpha^{k-1}\beta^{j+1-k}\\
=&\begin{cases}(-1)^j\cdot \beta^j\cdot (b-\beta)&\mbox{ if } k=0 \\ (-1)^j\alpha^{k-1}\beta^{j-k}\cdot \Big(b\binom{j}{k}\alpha-\binom{j+1}{k}\alpha \beta+a\binom{j}{k-1}\beta\Big)&\mbox{ if }0<k<j+1\\ (-1)^j\alpha^j(a-\alpha)&\mbox{ if } k=j+1\end{cases}\\
=& \begin{cases}(-1)^j\cdot \beta^j\cdot (b-\beta)&\mbox{ if } k=0 \\ (-1)^j\alpha^{k-1}\beta^{j-k}\cdot \Big(\alpha\binom{j}{k}(b-\beta)+\beta\binom{j}{k-1}(a-\alpha)\Big)&\mbox{ if }0<k<j+1\\ (-1)^j\alpha^j(a-\alpha)&\mbox{ if } k=j+1.\end{cases}
\end{align*}

For $x^iz$, we obtain
$$(j+1)\cdot (-1)^j\cdot \beta^j+b\cdot j\cdot (-1)^{j-1}\cdot \beta^{j-1}=\beta^{j-1}\cdot (-1)^j\cdot (bj-(j+1)\beta).$$

For $x^{i+j}z$, we have
$$(j+1)\cdot (-1)^j\cdot \alpha^j+a\cdot j\cdot(-1)^{j-1}\cdot \alpha^{j-1}=\alpha^{j-1}\cdot (-1)^{j-1}\cdot (aj-(j+1)\alpha).$$

No matter where $P_1$ is located, we pick the $\mathbb{C}$-coefficients of the Puiseux series $\alpha$, from the initial (i.e.\ from the coefficient of $t^0$) to order $\lambda_1$, to be equal to the corresponding coefficients of $a$, and the same for $\beta$ and $b$.
It follows that the valuation of the coefficient of $x^{i+k}$ is at least $\lambda_1$ for each $k=0,\ldots,j+1$, whereas the valuation of the coefficient of $x^iz$ and $x^{i+j}z$ is $0$.
To pick the coefficients of $\alpha$ and $\beta$ of order $\lambda_1$ and higher, we have to distinguish cases depending on the location of $P_1$.

Assume first that $P_1$ is at $x^iy^{j-1}$ (case (1) in Figure \ref{fig-Newtstarshaped}).
Notice that the contribution $x^i$ gets from $P_1$ by ``cone feeding'' is $d\cdot t^{\lambda_1}(-1)^{j-1}\beta^{j-1}$, if $d\cdot t^{}\lambda_1$ is the coefficient of the term of $P_1$ in $q$, where $d\in K$ has valuation $0$.
Thus, after picking $\alpha$ and $\beta$ up to order $\lambda_1$, the lowest order term of the coefficient of $x^i$ is
$$(-1)^{j-1}\cdot \beta_0^{j-1}\cdot (-(b_{\lambda_1}-\beta_{\lambda_1})\cdot \beta_0)+d_0,$$
where we denote by $b_i$ the coefficient of $t^i$ in $b$ and accordingly for $\beta$. Using $b_0=\beta_0$ we can make the above disappear by picking $\beta_{\lambda_1}=b_{\lambda_1}-\frac{d_0}{b_0}$.
We let $\alpha$ and $a$ agree at order $\lambda_1$.
Notice that this choice implies that the valuation of the coefficients of $x^{i+1},\ldots,x^{i+j-1}$ is exactly $\lambda_1$.
Continuing like this, we can pick $\alpha$ and $\beta$ such that the valuation of the coefficient of $x^i$ is greater than $\frac{1}{2}(\lambda_1+\lambda_2)$, and the valuation of the coefficient of $x^{i+j+1}$ is greater than $\frac{1}{2}(\lambda_1+\lambda_3)$. 
It follows from the smoothness that the valuation of the coefficient of $x^{i-1}$ is $\lambda_2$, and of $x^{i+j+2}$ is $\lambda_3$.
We conclude that the (relevant part of the) Newton subdivision (i.e.\ the local picture visible in the attached cell of the modification of $\RR^2$) of $\Trop(V(q_{xz}))$ is as depicted in Figure \ref{fig-Newtstarshaped2}. 

In particular, $\Trop(V(q_{xz}))$ has a unique liftable bitangent line $\Lambda'$  which locally at each tangency point satisfies the conditions of Lemma \ref{nonRegularIntersection}. Notice that it follows from our assumption $-\lambda_1>-\lambda_2\geq -\lambda_3$ that the two points of tangency are contained in the open interiors of the cells $\sigma_4$ resp.\ $\sigma_6$ of the modified plane.

\begin{figure}
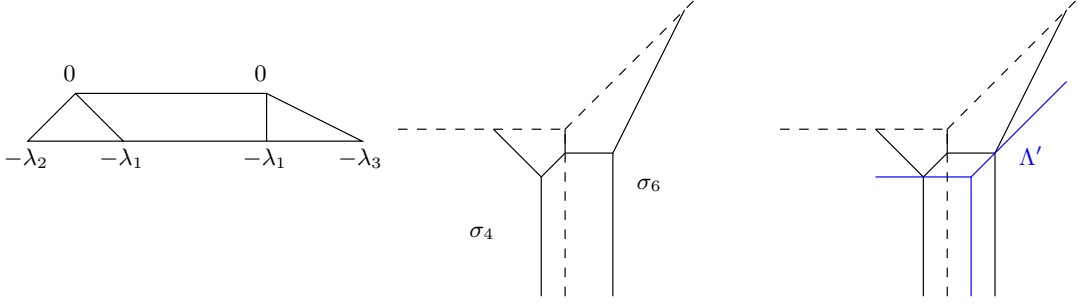
\caption{The relevant part of the Newton subdivision of the projection $q_{xz}$ revealing the features of the re-embedded tropical curve in the cells $\sigma_4$ and $\sigma_6$ of the modified plane (the numbers denote the negative valuations of corresponding coefficients). In the middle, the part of $\Trop(q_{xz})$ in question; on the right, the unique liftable tropical bitangent line $\Lambda'$. }
\label{fig-Newtstarshaped2}
\end{figure}

Using Lemma \ref{nonRegularIntersection}, we can now, for each point of tangency, pick two choices for the initials of the coefficients of $\Lambda'$. Combining these four choices with the values we picked for the low order terms of $\alpha$ and $\beta$ in the modification process, we obtain exactly four lifts of $\Lambda$.

If $P_1$ corresponds to the monomial $x^{i+1}y^{j-1}$ (case (2) in Figure \ref{fig-Newtstarshaped}), we proceed analogously:
We start by picking $\alpha_0=a_0,\ldots,\alpha_{\lambda_1-1}=a_{\lambda_1-1}$ and $\beta_0=b_0,\ldots,\beta_{\lambda_1-1}=b_{\lambda_1-1}$.
The valuation of the coefficient of $x^{i}$ is $\lambda_2$ and so the vertex dual to $x^{i-1}$, $x^{i}z$ and $x^{i+1}$ (whose coefficient has valuation $\lambda_1$) already satisfy the requirements. The coefficient of $x^{i+j+1}$ may inherit a contribution of order $\lambda_2$, and since $\lambda_2$ may be less than $\lambda_3$, we then have to pick $\alpha$ at the corresponding orders accordingly to make the negative valuation drop below the average of the negative valuation of the coefficient of $x^{i+j}$ (which is $-\lambda_1$) and $x^{i+j+2}$ (which is $-\lambda_3$). In any case, we will end up with the same picture as in Figure \ref{fig-Newtstarshaped2}, and thus with exactly four lifts of our bitangent.

The case (3) of Figure \ref{fig-Newtstarshaped}  ($P_1$ corresponding to $x^{i+2}y^{j-1}$) is symmetric to case (1) and works analogously. 

\end{proof}

\begin{remark} The modification technique used in the proof of Proposition \ref{prop-starshaped} can be adapted to the case where one of the edges of $\Gamma$ adjacent to $V$ is unbounded. In this case, the local intersection multiplicity is not $4$ but $3$, and there is only one tangency point $P$ near $V$. The other tangency point $Q$ has to be on a ray of $\ell$ which aligns with one of the bounded edges $e$ at $V$. Using the Genericity assumption \ref{genass} that the initial lift of $Q$ does not equal the initial lift at $e$ (i.e.\ the quotient of the initials of the two coefficients corresponding to $e$ in the defining equation for $\Gamma$), we conclude that $P$ must be the midpoint of the other bounded edge at $V$, and the local lifting multiplicity for $\Lambda$ is $m(C,\Lambda,P)=2$.

\end{remark}

\begin{figure}
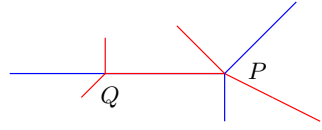
\caption{The stable intersection of $\Lambda$ and $\Gamma$ at the vertex $P$ of $\Lambda$ is $3$, the stable intersection at $Q$ is $1$.}
\label{fig-noliftforpointsinsegment}
\end{figure}

\begin{lemma}\label{lem-noliftforpointsinsegment}
Suppose that a connected component of $\Gamma\cap\Lambda$ is a line segment containing points $P$ and $Q$ (see Figure \ref{fig-noliftforpointsinsegment}). Then $(\Lambda,P,Q)$ doesn't lift to a bitangent of $C$. 
\end{lemma}

\begin{proof}
Since $P$ and $Q$ are in the same connected component, the stable intersection multiplicity between $\Gamma$ and $\Lambda$ along the component must be at least $4$.
As in the discussion at the beginning of the subsection, we may assume that one of the points, $P$ is at the vertex of the line. Further, we may assume that $Q$ is  the midpoint of the segment, since otherwise we already know that the bitangent doesn't lift.  
Assume that the vertex is at the origin. By symmetry we can assume that the overlapping edge is horizontal. By smoothness, the curve $q$ has the form $a+by+cxy^k + O(t)$ for some $|k|\geq 2$. 

We treat the case $k\geq2$. The case $k\leq-2$ follows similarly. 
Write $q = a+by+cxy^k +O(t)$. Let $\ell=y+m+nx$. 
In order for the  tangency point at the middle of the horizontal edge to lift, it is necessary that $\overline{m}$ equals the initial lift given in Definition \ref{def-initiallift}, namely\ $\overline{m}= \frac{\overline{a}}{\overline{b}}$ (see case (3) of Theorem \ref{Thm:Lifting}, resp.\ Proposition \ref{prop-horizontalsegment}). Considering the first order terms of the equations of the line, the curve and the Wronskian for the tangency point at the vertex (see Section \ref{sec-wronskian}), we get  
\[
\overline{y}+\frac{\overline{a}}{\overline{b}} + \overline{n}\overline{x}=0,\;\;
\overline{a}+\overline{by}+\overline{c}\overline{xy^k}=0,
\;\;
\overline{n}(\overline{b}+k\overline{cxy^{k-1}})-\overline{cy^k}=0.
\]
The first equation is equivalent to $\overline{nbx} = - \overline{by} -\overline{a}$, and the third is equivalent to $\overline{cxy^k}= \overline{nx(b+kcxy^{k-1})}$. Plugging both of these into the second equation, we get $\overline{(\frac{a}{b}+y)kcxy^{k-1}}=0$, so $\overline{y}=-\overline{\frac{a}{b}}$. But then the first equation becomes $\overline{nx}=0$, a contradiction. 

\end{proof}

\section{The $28=7\times 4$ tropical bitangents to a quartic}\label{sec-quartics}
This section is devoted to the proof of the following theorem:

\begin{theorem}\label{thm-quartics}
Let $\Gamma=\Trop(C)$ be a smooth tropicalization of a generic quartic curve (in the sense of the assumptions used in Theorem \ref{Thm:Lifting}, \ref{genass}, and Proposition \ref{prop-starshaped}). Then in each equivalence class of tropical bitangents, only finitely many representatives lift, and the sum of the multiplicities of those lifts is $4$.
Moreover, every partition of $4$ except $3+1$ is realizable as the lifting multiplicities of representatives of a bitangency class.
\end{theorem}

\begin{proof}
Fix a smooth tropicalized quartic $\Gamma$ satisfying the genericity condition, and let $P+Q$ be a theta characteristic.   By \cite{BLMPR}, there is always at least one bitangent $\Lambda$ between $P$ and $Q$.  If one of the connected components of the intersection is a bounded segment, then one of the points will be in the midpoint of that segment.
Note that we can rule out any intersection of multiplicity $3$ because then the line will not be bitangent.

Let us first deal with the case where the stable intersection along one connected component has multiplicity $4$.
Any non-proper such intersection must contain the vertex of $\Lambda$, and so it is symmetric to one of the cases considered in Proposition \ref{prop-starshaped} and Lemma \ref{lem-noliftforpointsinsegment}. In Proposition \ref{prop-starshaped}, the line $\Lambda$ is easily seen to be the unique element in its equivalence class, and it lifts with multiplicity $4$.
In Lemma \ref{lem-noliftforpointsinsegment}, $\Lambda$ itself does not lift, but as we see in case (a3) below there are two other representatives of its equivalence class that each lifts with multiplicity $2$.

A proper intersection of multiplicity $4$ can (up to symmetry) only occur on an edge of slope $(4,1)$. A line intersecting the open interior of such an edge cannot lift to a bitangent by Theorem \ref{Thm:Lifting}. Assume the vertex of $\Lambda$ intersects a vertex of $\Gamma$ adjacent to an edge of slope $(4,1)$ at $P=(0,0)$. By smoothness, $P$ has to be adjacent to a horizontal edge of $\Gamma$. That is, its dual polygon can be assumed to have vertices $(1,0)$, $(0,4)$ and $(1,1)$, or $(1,0)$, $(0,4)$ and $(0,3)$.
There must be either two different tangency points tropicalizing to $P$ or one point that admits a hyperflex, i.e.\ a line intersecting $C$ at one point with multiplicity $4$.
Let us first assume there are two different tangency points, and that the polygon is the first choice above. That is, there are points $(x,y)$ and $(r,s)$ whose initials solve the following system of equations:
\[
\overline{ax}+\overline{by^4}+\overline{cxy}= 0, \;\; \overline{y}+\overline{m}+\overline{nx} = 0,\;\;  (4\overline{by^3}+\overline{cx})\cdot \overline{n}-(\overline{a}+\overline{cy}) =0,
\]
\[
\overline{ar}+\overline{bs^4}+\overline{crs}= 0, \;\; \overline{s}+\overline{m}+\overline{nr} = 0,\;\;  (4\overline{bs^3}+\overline{cr})\cdot \overline{n}-(\overline{a}+\overline{cs}) =0,
\]
as in Section \ref{sec-wronskian}.
We use \textsc{Singular} \cite{DGPS} to compute the saturation of the ideal generated by these polynomials by the ideal generated by $x-r$ and $y-s$. In this way, we impose that the two tangency points lifting $P$ are distinct.
The saturation of the ideal is generated by
\[
-4\overline{a^2bn}+\overline{c^3},\;-\overline{cm}-\overline{a},\;
\overline{c^4r}+4\overline{a^2bcs}-4\overline{a^3b},\;\overline{y}+\overline{s},\;\overline{c^4x}-4\overline{a^2bcs}-4\overline{a^3b},\;-\overline{c^2s^2}+2\overline{a^2}.
\]
It follows that there is exactly one lift of $\Lambda$ to a bitangent of $C$. The computation for the other possible polygon is analogous and also leads to one lift. 
A similar computation with \textsc{Singular} shows that $\Lambda$ does not lift to a hyperflex.

Assume now that $\Lambda$ has its vertex at the point dual to the polygon with vertices $(1,0)$, $(0,4)$ and $(0,3)$. We can move the vertex of the tropical line along the edge of slope $(4,1)$ until we obtain $\Lambda'$ whose vertex is at the point dual to the polygon with vertices $(1,0)$, $(0,4)$ and $(1,1)$. We vary the tropical line within its equivalence class. We can continue moving the vertex of the tropical line along the adjacent horizontal bounded edge (staying within the equivalence class of the bitangent), until we obtain $\Lambda''$ whose diagonal end hits another vertex (see Figure \ref{fig-intfour}). By smoothness of $\Gamma$, this is the only way we can continue to vary the tropical bitangent within its equivalence class. The bitangent $\Lambda''$ has two tropically different tangency points: one in the middle of the horizontal line segment with local lifting multiplicity two according to Theorem \ref{Thm:Lifting}, and one at a vertex of $\Gamma$ with local lifting multiplicity $1$. It follows that $\Lambda''$ lifts with multiplicity $2$. Obviously, we leave the equivalence class if we continue moving the vertex of the tropical bitangent in any other way.
Thus, for this case we have two members of the equivalence class that lift once and one member that lifts twice.

\begin{figure}
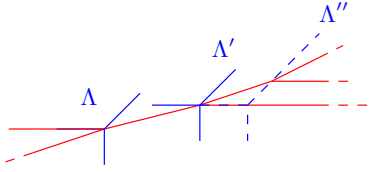
\caption{An equivalence class for which three members lift nontrivially.}
\label{fig-intfour}
\end{figure}

Let us now assume as in Theorem \ref{Thm:Lifting} that $P$ and $Q$ are contained in different connected components of $\Lambda \cap \Gamma$.
We begin by considering the case where $P$ and $Q$ are in the interior of different ends of the line.
Without restriction, we assume that $P$ is on the horizontal end of $\Lambda$, and $Q$ on the diagonal.
We fix $P$ and move $\Lambda$ in a one-dimensional family by moving its vertex on a horizontal line, studying the different possibilities for the second tangency point $Q$.
Possible locations of $Q$ with respect to $\Gamma$ are:
\begin{enumerate}
\item[(a)] $Q$ can be at a vertex $v$ of $\Gamma$.
\item[(b)] $Q$ can be in the interior of an edge of $\Gamma$.
\end{enumerate}
In the second case, there are two possibilities to study:
\begin{enumerate}
\item[(b1)] Either the edge of $\Gamma$ containing $Q$ meets the diagonal end of $\Lambda$ transversally,
\item [(b2)] or $\Gamma$ and $\Lambda$ intersect in a segment containing $Q$.
\end{enumerate}
In case (a), the local lifting multiplicity of $\Lambda$ at $Q$ is $1$ by Theorem \ref{Thm:Lifting}. 
The vertex $v$ is adjacent to three edges, of which exactly one, say $e$, meets the diagonal end of $\Lambda$ with multiplicity $2$. It has to be a bounded edge. The one-dimensional family of tropical lines we obtain by moving $\Lambda$ (fixing $P$) has a bounded subfamily where the intersection of the diagonal end is along $e$. The two boundary points of the subfamily correspond to $\Lambda$ and
a tropical line $\Lambda'$ satisfying one of the following:
\begin{enumerate}
\item[(a1)] The diagonal end of $\Lambda'$ intersects the other vertex $Q'$ of $e$ (see Figure \ref{fig-combinations} (a1)); in this case the local lifting multiplicity of $\Lambda'$ at $Q'$ is $1$ as before by Theorem \ref{Thm:Lifting}.
\item[(a2)] The vertex $Q'$ of $\Lambda'$ intersects $e$ (see Figure \ref{fig-combinations} (a2)). Using Theorem \ref{Thm:Lifting}, we can deduce that the local lifting multipliticy of $\Lambda'$ at $Q'$ is $|\lambda|$ if $(\mu,\lambda)$ is the direction of $e$. Since $e$ is an edge of a quartic, and since it meets the diagonal end of $\Lambda$ with multiplicity $2$,  $(\mu,\lambda)$ can be $(-3,-1)$, $(1,-1)$ or $(1,3)$. 
We want to exclude the case $(\mu,\lambda)=(1,3)$. In the dual Newton subdivision, there are only three possible positions for the dual of an edge of direction $(1,3)$ (see Figure \ref{fig-Bcase13}, left).
For two of these positions (drawn as green lines in the picture), it is immediate to see that a horizontal end of $\Lambda'$ starting at a point of the edge $e$ cannot have any further intersection with $\Gamma$, so this position is not possible in our situation. For the third position (drawn as red line in the picture), we would have a bounded connected component of $\mathbb{R}^2\setminus\Gamma$ through which the horizontal end of $\Lambda'$ would pass. It is easy to see checking the possible directions of edges bounding this component that we could not have a further tangency point on the horizontal end of $\Gamma$ (see Figure \ref{fig-Bcase13}, right). Hence also this position is not possible in our situation and we have now excluded the case that $(\mu,\lambda)=(1,3)$. Thus, in any case, the local lifting multiplicity of $\Lambda'$ at $Q'$ is $|\lambda|=1$.  

\begin{figure}
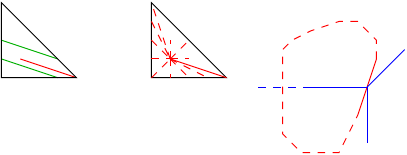
\caption{Excluding the direction $(1,3)$ for an edge of a quartic meeting a vertex of a bitangent which has another intersection of multiplicity $2$ on its horizontal end.}
\label{fig-Bcase13}
\end{figure}

\item[(a3)] The vertex $Q'$ of $\Lambda'$ intersects the other vertex of $e$ (see Figure \ref{fig-combinations} (a3)). Let us study possible dual polygons of the vertex of $\Gamma$ meeting the vertex of $\Lambda'$: as above, the edge $e$ must have direction $(-3,-1)$, $(1,-1)$ or $(1,3)$. The adjacent vertex has to be smooth. The local intersection multiplicity at the vertex either equals three if it is adjacent to the horizontal edge containing the other tangency point (as in Lemma \ref{lem-noliftforpointsinsegment}, see Figure \ref{fig-noliftforpointsinsegment}), or it cannot be bigger than two.
In the first case, we can continue moving the vertex of $\Lambda'$ to the right until we obtain $\Lambda''$ whose vertical end hits another vertex $Q''$ of $\Gamma$ as in case (a1). The local lifting multiplicity of $\Lambda''$ at $Q''$ is $1$.
In the second case, the dual polygon must be (a shift of) the triangle with vertices $(0,0)$, $(1,0)$ and $(1,1)$. To compute the local lifting multiplicity, we can use Proposition \ref{prop-vertexofboth} after a suitable coordinate change, resp.\ a direct computation. After a coordinate change, the curve has rays $\binom{-1}{0}\binom{-1}{-1},\binom{2}{1}$, which is the case $i=1,j=1,k=0,l=2$ of Proposition \ref{prop-vertexofboth}. It follows that the local lifting multiplicity of $\Lambda'$ at $Q'$ is $1$.
\end{enumerate}

In case (b1), $\Lambda$ itself does not lift, but it belongs to a bounded subfamily as in case (a) for which the two boundary elements lift just as before (see Figure \ref{fig-combinations} (b1)).

In case (b2), $\Lambda$ cannot be varied in a nontrivial family fixing $P$ and maintaining a second tangency point. The local lifting multiplicity at $Q$ is two (see Figure \ref{fig-combinations} (b2)).

In any case, we see that we can vary $\Lambda$ in a (possibly trivial) family of tropical lines such that the sum of the local lifting multiplicities is $2$.
By symmetry, we obtain the same by fixing $Q$, and thus altogether we get a family such that the sum of the multiplicities of all members that lift equals $4$.

\begin{figure}
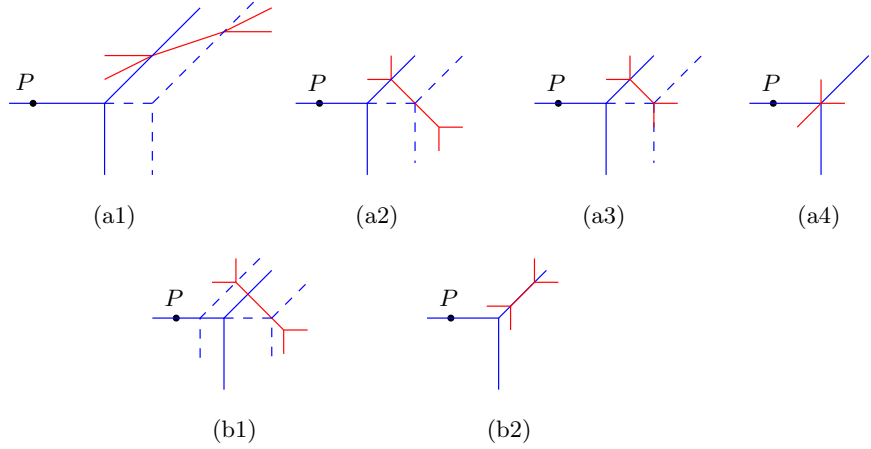
\caption{Possible locations of tangency points.}
\label{fig-combinations}
\end{figure}

If $P$ and $Q$ are in the interior of the same end of the line, then by the genericity assumption the line $\Lambda$ does not lift (this includes the case where a segment extends all the way to the vertex). 
By moving the vertex of $\Lambda$ until one of the tangency points is at its vertex, we get a bitangent that possibly lifts, and we deal with this case next. 

If $P$ is in the interior of (without restriction) the horizontal end, and $Q$ is at the vertex of $\Lambda$, then we have to consider the following cases:

\begin{enumerate}
\item[(a2')] As in case (a2) above, $Q$ can be in the interior of an edge of $\Gamma$. We can vary $\Lambda$ in a one-dimensional family maintaining tangency until the diagonal end meets a vertex as in Figure \ref{fig-combinations} (a2). As in case (a2), for both boundary points of the family, the local lifting multiplicity of $\Lambda$ (resp.\ $\Lambda'$) at $Q$ (resp.\ $Q'$) is one.
\item[(a3')] $Q$ can be at a vertex of $\Gamma$ which intersects $\Lambda$ non-transversally. 
If we can move $\Lambda$ maintaining the tangency, then this is the behaviour we have studied in case (a3) before.
As we have seen in (a3), there is only one possible dual polygon for such a vertex. There, we have also computed the local lifting multiplicity of $\Lambda$ at $Q$ to be one. As above, $\Lambda$ can be varied in a bounded one-dimensional family such that the other boundary point is the only one with nontrivial local lifting multiplicity at the respective intersection point, and the local lifting multiplicity equals one.

If we cannot move $\Lambda$ maintaining the tangency, then there is only one case to check up to symmetry, namely that the vertex of $\Gamma$ is dual to a polygon with vertices $(1,0)$, $(2,0)$ and $(0,1)$. This is true, since we can without restriction assume that the vertical end of $\Lambda$ overlaps with an edge of $\Gamma$, and since we have an intersection of multiplicity two. So, we can assume $q=ax+bx^2+cy+O(t)$, where $a$, $b$ and $c$ have valuation zero. Using Proposition \ref{prop-vertexofboth} (after a suitable coordinate change), resp.\ a direct computation, we conclude that the local lifting multiplicity is two.
An example for this case can be seen in Figure \ref{fig-bigquartic}; it is the line with vertex at the point labeled 3.

\item[(a4)] Finally, $Q$ can be at a vertex of $\Gamma$ at which $\Gamma$ and $\Lambda$ intersect transversally (see Figure \ref{fig-combinations} (a4)). Then locally we are in case (5) of Theorem \ref{Thm:Lifting}. $\Lambda$ does not vary in a (nontrivial) family maintaining the tangency: if we move its vertex left, it is not locally tangent, if we move it to the right, the two tangency points are in the interior of the same end of the tropical line which therefore does not lift by Theorem \ref{Thm:Lifting}. The local lifting multiplicity of $\Lambda$ at $Q$ is $2$ by Theorem \ref{Thm:Lifting}.
\end{enumerate}
\noindent As before, we can see that a tropical bitangent varies in a (possibly trivial) bounded family for which only the boundary points have a nontrivial local lifting multiplicity, and the sum of the local lifting multiplicities equals $2$ in each case.

\end{proof}

\begin{example}\label{ex-quartic}
To illustrate the main result of this section, we find the liftable bitangents of an explicit tropical quartic. Figure \ref{fig-bigquartic} shows a tropicalized quartic satisfying the requirements of Theorem \ref{thm-quartics} (with defining equation (\ref{eq-exq})), and the vertices of the tropical bitangent lines that lift. Note that the vertices labelled $2a$ and $2b$ belong to tropical lines in the same equivalence class, they each lift with multiplicity $2$. All other vertices belong to the unique member of the respective equivalence class that lifts, and they each lift with multiplicity $4$.
\begin{figure}
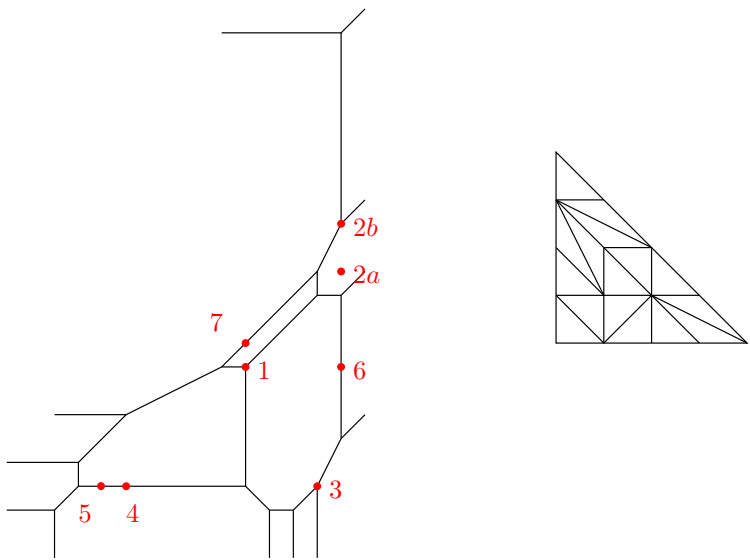
\caption{A plane tropical quartic and the vertices of its seven bitangent lines. On the right, the dual Newton subdivision.}
\label{fig-bigquartic}
\end{figure}
\end{example}

\begin{remark}
As exhibited by Akshay Tiwary and Wanlong Zheng during a FUSRP (Fields Undergraduate Summer Research Program) at the Fields Institute \cite{TZ}, there are smooth tropical quintic and sextic curves such that the number of lifts is not the same for all the equivalence classes of bitangents. In particular, it is not true anymore that every bitangent class lifts $4$ times. It is still conjectured that the number of lifts of every bitangent class is divisible by $4$, and that there is a finer equivalence relation for which every class lifts $4$ times. 
\end{remark}


\section {Lifting real bitangents}\label{sec-realbitangents}

The question of how many bitangent lines of a real quartic are real dates back to Pl\"ucker. The answer depends on the topology of the real curve, a topic studied intensely in real algebraic geometry in the context of Hilbert's sixteenth problem. A real quartic can have between $0$ and $4$ ovals. It has $4$ real bitangents if it has $0$, $1$ or $2$ nested ovals, $8$ if it has $3$ non-nested ovals, $16$ if it has $3$ ovals and $28$ if it has $4$ ovals \cite{PSV}. 

By the Tarski principle \cite{JL89}, the field of Puiseux series with real coefficients is equivalent to the reals, so we can count real bitangents by lifting tropical bitangents to $L$. This principle has been applied to other problems in tropical geometry, see e.g.\ \cite{ABGJ, MMS15}. By Viro's Patchworking method, we can read off the topology of a real quartic (close to the tropical limit) from its tropicalization and the signs of the coefficients \cite{Viro, IMS09}. 

The question which of the $7$ tropical bitangent classes lifts, and how many lifts it has, can be answered for concrete cases with our lifting approach. Thus we lay the groundwork for a tropical enumeration of real bitangents. Before we can deduce Pl\"ucker's real bitangent count using tropical geometry, a major combinatorial classification task has to be completed. The following example illustrates the applicability of our methods to tropical counts of real bitangents. 

As in the proof of Theorem \ref{thm-quartics}, the local (complex) lifting multiplicity at each tangency point of a tropical bitangent was either $1$ or $2$. If a local lifting multiplicity is $1$, then the unique solution has to be real if the input data is real. If a local lifting multiplicity is $2$, our solutions for initials involve a square root, and so there can be either $2$ or no real solution, depending on the sign of the radicand.


\subsection{A comprehensive example for real lifts of tropical bitangents to a tropicalized quartic}
Consider the quartic from example \ref{ex-quartic}, given by equation (\ref{eq-exq}) over the reals. We now go through the seven tropical bitangents depicted in Figure \ref{fig-bigquartic} and determine the number of real lifts for each. We denote the tropical bitangent with vertex at $i$ by $\Lambda_i$.

For $\Lambda_1$, we have to use the modification technique (II) described in Subsection \ref{subsec-modification}. We use the equation $z-x-y+ty+t^2y-1$ to re-embed. Projecting to the $yz$-coordinates, we obtain a curve with two tangency points as in Lemma \ref{nonRegularIntersection} that each give two lifts for the two coefficients of the bitangent. 
The (relevant part of the) projection of the re-embedded curve to the $yz$-plane is depicted in Figure \ref{fig-bitangent1modified}.
\begin{figure}
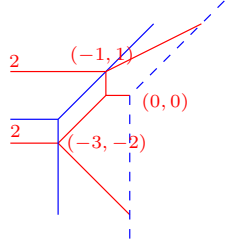
\caption{The re-embedded curve with equation (\ref{eq-exq}), modified at bitangent $1$ and projected to the $yz$-plane. In blue, the bitangent line for which we compute lifts.}
\label{fig-bitangent1modified}
\end{figure}
We compute the lifts for the tangency point $(-1,1)$. The local equation of the tropicalized curve is $q=-t+t^3y^2-z$, the local equation of the line is $\ell=z+k\cdot y$, and so for the Wronskian as in Subsection \ref{sec-wronskian} we get $W=2t^3y+k$. Solving these equations, we obtain a negative radicand for $y$, which yields two non-real solutions for the coefficient $k$ of the modified bitangent line. Thus there are no real lifts for $\Lambda_1$.

For $\Lambda_{2a}$ and $\Lambda_{2b}$, we use the modification technique (I) described in Subsection \ref{subsec-modification}. Passing to the projection of interest, we can immedialy substitute $x$ with $x-\frac{1}{3t^4}$. The solutions for the initials we obtain from Lemma \ref{nonRegularIntersection} involve the square root of the product of the constant term and the $y^2$-coefficient after this substitution, whose initials are $1/81$ resp.\ $1/9$, so the radicand is positive and we get two real solutions for the $x$-coefficient of $\ell$ in each case. Since the bitangents with vertices at $2a$ and $2b$ each have local lifting multiplicity one at the other tangency point, we get real solutions for the constant coefficient of $\ell$ also and in conclusion, the bitangent class $2$ with the two bitangents $\Lambda_1$ and $\Lambda_2$ lifts to 4 real bitangents.

For $\Lambda_3$, we use the modification technique (I), substituting $y$ with $y-t^5$. The new constant coefficient has leading term
 $-t^{12}$, the new $x^2$-coefficient has leading term $-t^5$, and the new $xy$-coefficient has leading term $1$. Using the solutions computed in Lemma \ref{nonRegularIntersection} we obtain for the second term of $m$ the two real solutions $\pm 2 t^{\frac{17}{2}}$. We use this to set up the local equations for computing lifts of $\ell$ locally around the second tangency point:
\begin{displaymath}
q=t^{11}x^4+t^8x^3+x^2y\,;\; \ell=y+t^5\pm 2t^{\frac{17}{2}}+nx;\;\;W=2t^{11}x+t^8-n.
\end{displaymath}
Solving for $x$, $y$ and $n$ we obtain two imaginary solutions for $n$, so $\Lambda_3$ has no real lift.

For $\Lambda_4$, we substitute $y$ with $y-x-tx$. Then the $x^4$-coefficient starts with $t^3$, the $x^2$-coefficient with $-1$, the $x^2y$ with $1$ and the $x^3$ with $-2t^2$ (which is not visible). We can thus use the solutions computed in Lemma \ref{nonRegularIntersection}, and obtain two non-real solutions. Thus, $\Lambda_4$ does not have any real lifts.

For $\Lambda_5$, we substitute $y$ with $\frac{1}{t}\cdot (y-x+tx+t^2x)$. Then the $x^4$-coefficient starts with $t^3$, the $x^2$-coefficient with $-t$ and the $x^2y$-coefficient with $3t^4$ (with the $x^3$-coefficient not visible again). We can use the solutions computed in Lemma \ref{nonRegularIntersection}, and obtain two non-real solutions. Thus, $\Lambda_5$ does not have any real lifts. 

For $\Lambda_6$, we need to use the modification technique (II). We use the equation $z-3t^4x-y-1$ to re-embed. Projecting to the $yz$-coordinates, we obtain a curve with a tangency point as in Lemma \ref{nonRegularIntersection}, with local equation $q=t^3-9y^2+xy$, so using the lifting equations from Lemma \ref{nonRegularIntersection} we can see that $\Lambda_6$ does not have any lift, since the radicand is $1\cdot (-9)$.

For $\Lambda_7$, we use modification technique (II) with the equation $z-(x+1+ty+t^2y)$. Projecting to the $yz$-coordinates, we obtain a curve with two tangency points as in Lemma \ref{nonRegularIntersection}, of which one has local equation $q=t^5 y^4-y^2+xy^2$, so the radicand is $(-1)\cdot 1$ and the bitangent has no real lift.
Altogether we have $4$ real lifts for the tropical bitangents of the tropicalization of the quartic defined by equation (\ref{eq-exq}) --- the $4=2+2$ lifts of $\Lambda_{2a}$ and $\Lambda_{2b}$ in Figure \ref{fig-bigquartic}.

We compare this to the answer expected by Viro's patchworking method. Figure \ref{fig-Viro} shows the result of the combinatorial patchworking method for equation (\ref{eq-exq}). For our computations, we used the online tool computing the topology of a patchworked curve (authored by B. El-Hilany, J. Rau and A. Renaudineau) 
that can be found at 
\begin{displaymath}\mbox{https://www.math.uni-tuebingen.de/user/jora/patchworking/patchworking.html}.\end{displaymath} It builds on the Viro.Sage package \cite{Viro:Sage}.

\begin{figure}
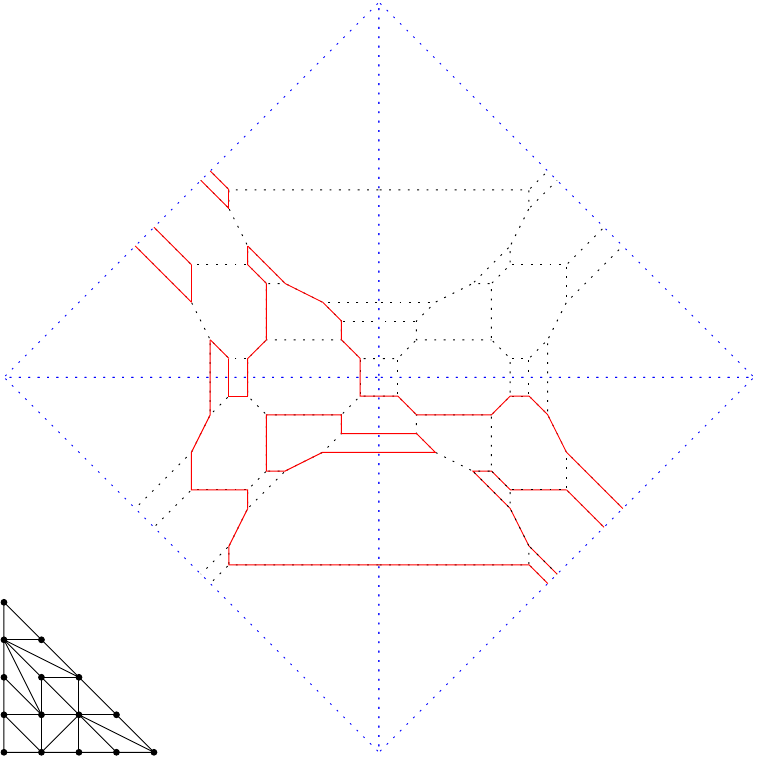
\caption{Combinatorial patchworking for the curve defined by equation (\ref{eq-exq}) reveals two nested ovals.}
\label{fig-Viro}
\end{figure}

We can see that a real curve close to the tropical limit has two nested ovals, so by \cite{PSV} we expect $4$ bitangents. This confirms our lifting computation above.

\subsection{Further study}

We can deduce from the example above that the number of real lifts of a tropical bitangent cannot be computed solely from the combinatorial properties of the tropicalized curve: the signs of the coefficients play a role in the solution, but even more challenging, the signs after coordinate changes (which we determine using modifications) have to be controlled. Before our methods can be souped up to a tropical way to count real bitangents, this obstacle has to be overcome, and a combinatorial classification of real lifts has to be mastered.

Nevertheless, in all our example computations, we observed that a tropical bitangent class,  either has $4$ real lifts (i.e.\ all complex lifts are real), or no real lift. This was true even in cases where a bitangent class had several liftable representatives, e.g.\ $3$ representatives that lift with (complex) lifting multiplicity $1$, $1$ and $2$. (A priori one could be tempted to think that we just need to pick signs in such a way that the representative with complex lifting multiplicity $2$ has only imaginary solutions --- that way we would produce an example with a class which has $1+1+0=2$ lifts --- but in our example classes this was not possible.)

Considering that the number of real bitangents to a quartic is always divisible by $4$, and building on our examples and computations, we set up the following conjecture:

\begin{conjecture}
A tropical bitangent class of a tropicalization of a generic smooth quartic defined over the real Puiseux series has either $0$ or $4$ real lifts.
\end{conjecture}


\end{document}